\documentclass[11pt,reqno]{amsart}
\usepackage{amssymb, hyperref}
\def\normo#1{\left\|#1\right\|}
\def\normb#1{\Big\|#1\Big\|}

\def\norm#1{\|#1\|}

\def\wt#1{\widetilde{#1}}
\def\wh#1{\widehat{#1}}
\def\set#1{\{#1\}}

\newcommand{\R}{{\mathbb R}}

\newcommand{\Z}{{\mathbb Z}}
\newcommand{\ft}{{\mathcal{F}}}

\newcommand{\les}{{\lesssim}}

\newcommand{\lag}{{\langle}}
\newcommand{\rag}{{\rangle}}

\newcommand{\Sch}{{\mathcal{S}}}

\numberwithin{equation}{section}
\newtheorem{theorem}{Theorem}[section]
\newtheorem{proposition}[theorem]{Proposition}
\newtheorem{lemma}[theorem]{Lemma}
\newtheorem{corollary}[theorem]{Corollary}
\newtheorem{remark}[theorem]{Remark}

\newcommand{\px}{\partial_x}
\newcommand{\pt}{\partial_t}

\newcommand{\la}{\ensuremath{\mathbb{\lambda}}}
\newcommand{\La}{{\Lambda}}
\newcommand{\del}{\delta}
\newcommand{\vep}{\varepsilon}

\begin{document}
\title[the fifth-order KdV equation]{Rough solutions of the fifth-order KdV equations}
\author[Z. Guo]{Zihua Guo}
\email{zihuaguo@math.pku.edu.cn}
\address{School of Mathematical Science, Peking University,
Beijing 100871, China. \\ Beijing International Center for Mathematical Research, Beijing 100871, China}

\author[C. Kwak]{Chulkwang Kwak}
\email{ckkwak@kaist.ac.kr}
\address{Department of Mathematical Sciences, Korea Advanced Institute of Science and Technology,
291 Daehak-ro Yuseong-gu, Daejeon 305-701, South Korea}

\author[S. Kwon]{Soonsik Kwon}
\email{soonsikk@kaist.edu}
\address{Department of Mathematical Sciences, Korea Advanced Institute of Science and Technology,
291 Daehak-ro Yuseong-gu, Daejeon 305-701, South Korea}

\begin{abstract}
We consider the Cauchy problem of the fifth-order equation arising from the Korteweg-de Vries (KdV) hierarchy
\begin{equation*}
\begin{cases}
\pt u + \px^5 u + c_1\px u\px^2 u + c_2u\px^3 u = 0 \qquad x,t \in \R \\
u(0,x) = u_0(x) \qquad u_0 \in H^s(\R)
\end{cases}
\end{equation*}
We prove a priori bound of solutions for $H^s(\R)$ with $s\geq \frac 54$ and the local well-posedness for $s \ge 2$.\\
The method is a short time $X^{s,b}$ space, which is first developed by
Ionescu-Kenig-Tataru \cite{IKT} in the context of the KP-I equation.
In addition, we use a weight on localized $X^{s,b}$ structures to reduce the contribution
of high-low frequency interaction where the low frequency has large
modulation.

As an immediate result from a conservation law, we have the fifth-order equation in the KdV hierarchy,
$$\partial_t u - \px^5 u -30u^2\px u + 20\px u\px^2 u + 10u\px^3u=0$$
is globally well-posed in the energy space $H^2$.
\end{abstract}

\thanks{}
\thanks{} \subjclass[2000]{35Q53} \keywords{ local well-posedness, the fifth-order KdV equation, KdV hierarchy, $X^{s,b}$ space. }

\maketitle

\tableofcontents

\section{Introduction}
We consider the Cauchy problem for the fifth-order KdV equation
\begin{equation}\label{eq:5kdv}
    \begin{cases}
\pt u + \px^5u + c_1\px u\px^2u + c_2u\px^3u = 0  \qquad  x,t \in \R\\
u(0,x) =u_0(x)    \qquad \qquad u_0 \in H^s(\mathbb{R})
    \end{cases}
\end{equation}
The fifth-order KdV type equation \eqref{eq:5kdv} generalizes the
second equation appearing in the KdV hierarchy:
\begin{equation}\label{eq:5kdv heirarchy}
  \partial_t u - \px^5 u -30u^2\px u + 20\px u\px^2 u + 10u\px^3u=0.
\end{equation}
Due to the theory of complete integrability, the KdV equation and
the higher-order equations in the hierarchy are solved via the
scattering and the inverse scattering methods at least for regular
and well-decaying initial data. Moreover, they enjoy infinitely many
conservation laws. However, the theory of complete integrability is
rigid so that one cannot apply to non-integrable equations. For
instance, if one slightly change a coefficient of the equation, then
the inverse scattering transform does not work. We note that there
are some other physical background of the equation \eqref{eq:5kdv}
such as higher-order water wave models, a lattice of an harmonic
oscillators. See \cite{Ponce94} for further discussion.

The purpose of this article is to study the Cauchy problem for
Sobolev initial data with low regularity in analytic manner.
Previously, Ponce \cite{Ponce94} showed \eqref{eq:5kdv} is locally
well-posed for $ s \ge 4$. Later, the third author \cite{Kwon}
improved the local well-posedness for $s> \frac 52$. Both works are
based on the energy method and local and global smoothing estimates
from dispersive effects. Furthermore, in \cite{Kwon}, it is shown
that the flow map fails to be uniformly continuous on a bounded set
of initial data for any $s \in \R$.   So, the Picard iteration
method cannot apply and such a less perturbative way is
necessary.\footnote{If one consider a weighted Sobolev spaces, then
Picard iteration may work. See, for example, \cite{KPV94}.} This is
a sharp contrast to the KdV equation, which is solved via the Picard
iteration \cite{KPV}. The issue is a strong low-high frequencies
interaction in the nonlinearity. Since the quadratic nonlinearity
has too many derivatives, they cannot be overcome by the dispersive
effect of linear part. As a result the equation shows a quasilinear
dynamics. This type of phenomenon is observed in other equations,
such as the Benjamin-Ono equation and the KP-I equation. Now we
state our main results:

\begin{theorem}\label{thm:main}
(a) Let $s\geq 5/4$, $R>0$. For any $u_0\in B_R=\{f\in H^\infty:
\norm{f}_{H^s}\leq R\}$, there exists a time $T=T(R)>0$ and a unique
solution $u=S_T^\infty (u_0)\in C([-T,T],H^\infty)$ to the
fifth-order KdV equation \eqref{eq:5kdv}. In addition, for any
$\sigma\geq 5/4$
\begin{eqnarray}
\sup_{|t|\leq T} \norm{S^\infty_T(u_0)(t)}_{H^{\sigma}}\leq
C(T,\sigma,\norm{u_0}_{H^{\sigma}}).
\end{eqnarray}

(b) Let $s\geq 2$. The map $S_T^\infty:B_R\to C([-T,T],H^\infty)$
extends uniquely to continuous mapping
\[S_T: \{f\in H^s:
\norm{f}_{H^s}\leq R\}\to C([-T,T],H^s).\]
\end{theorem}

We have a priori bound of solutions for $H^s(\R)$ with $s \ge \frac 54$, while the well-posedness holds for $s \ge 2$. 
For the proof of the theorem, we use $X^{s,b}$ type
structure in a short time interval depending on frequencies. This is
first developed by Ionescu, Kenig, and Tataru \cite{IKT} in the
context of KP-I equation, see \cite{CCT} for a similar idea and
\cite{KoTz2} in the setting of Strichartz norms. The method is a
combination of modified Bourgain's $X^{s,b}$ space and the energy
method.

First, observe that the bilinear $X^{s,b}$-estimates ($ \|uv_{xxx}
\|_{X^{s,b-1}} \lesssim \|u\|_{X^{s,b}}\|v\|_{X^{s,b}}$) fails in
usual $X^{s,b}$ for any $s$, and the difficult term is the high-low
interaction component with very low frequency of the following type
\[(P_{\leq 0}u)\cdot (P_{high}v_{xxx})\]
where the low frequency is very small such that these interactions
have a small resonance and coherence. But if one use $X^{s,b}$
structure for a short time interval $(\approx
\text{(frequency)}^{-2} )$, the contribution of high frequency and
low modulation is reduced so that we can prove the bilinear estimate
(See Remark~\ref{re:counterexample}).

To compensate this short-time estimates, we need an energy-type
estimates. In fact, we could not close the energy estimate solely
using Ionescu-Kenig-Tataru's method. The enemy is the high-low
interactions where the low frequency component has the largest
modulation. In \cite{GPWW}, the first author \emph{et. al.} used a
weight that was first used in \cite{IK} to strengthen estimates for this
interaction. It turns out that we have to use the weight differently on low and high frequencies.(See \eqref{beta weight}) Intuitively, we put more modulation regularity for the
low frequency component $P_{\leq 0}u$. Then this helps improving
the high-low interaction and the energy estimates.

The trade-off of using such a weight is to worsen the $high-high \to
low$ interactions in the nonlinear estimates
\[P_{\leq 0}(P_{high}u\cdot P_{high}v_{xxx}).\]
Fortunately, after rewriting the nonlinear term in the divergence
form $c_1\px u\px^2u + c_2u\px^3u=c_1'\partial_x(\px u \px
u)+c_2'\px(u\px^2u)$, we are able to choose a weight to balance both
purposes.

Around the time when we completed this work, we learned Kenig and Pilod
\cite{kenig-pilod} have worked on the same problem with a similar
idea an obtained the same result. They used the short-time $X^{s,b}$ structure and the modified
energy method. The modified energy is an energy norm with cubic
correctional terms that make a cancellation and so improve the
energy bound. \\
There are many works on similar types of higher-order dispersive
equations. For instance, the Kawahara equation, which is the
fifth-order equation with nonlinearity $uu_x$ \cite{BL} and the
fifth-order equation in the modified KdV hierarchy, which has the
cubic nonlinearities such as $u^2u_{xxx}$ \cite{Kwon2}. As opposed
to \eqref{eq:5kdv}, the dynamics of these equations are semi-linear
in the sense that the flow map is locally Lipschitz continuous and
so solved via the Picard iteration, since nonlinear feedback of
these equations are weaker.

Combined with the second conservation law in the KdV hierarchy,
$$ H_2(u) = \int \frac 12 (\partial_x^2u)^2 -5u\partial_x (u^2)  +   \frac 52 u^4 \,dx,
$$
 we can obtain the global well-posedness for the equation \eqref{eq:5kdv heirarchy}.
\begin{corollary}\label{cor:gwp}
  The Cauchy problem \eqref{eq:5kdv heirarchy} is globally well-posed in $H^2$.
  \footnote{\eqref{eq:5kdv heirarchy} have an additional cubic term. The nonlinear
  estimate and the energy estimate for this term are essentially easier.
  Thus, we only sketch the proof. See Remark \ref{re:cor1.2-bilinear} and \ref{re:cor1.2-energy}.}
\end{corollary}

The paper is organized as follows: In Section 2, we present notations and define function spaces.
In Section 3 and 4, we prove the bilinear estimates, the energy estimates.
In Section 5, we sketch the proof of well-posedness. We collect the proofs for known propositions in the appendix for convenience of readers.

\noindent{\bf Acknowledgement.} We appreciate Didier Pilod for pointing out error in the first draft. Z.G. is partially supported by NNSF
of China (No. 11001003) and RFDP of China (No. 20110001120111). C.K.
is partially supported by NRF(Korea) grant 2010-0024017. S.K. is
partially supported by TJ Park science fellowship and NRF(Korea)
grant 2010-0024017.

\section{Notations and Definitions}

For $x,y \in \R_+$, $x \lesssim y$ means that there exists $C>0$
such that $x \le Cy$. And $x \sim y$ means $x \lesssim y$ and $y
\lesssim x$. Let $a_1,a_2,a_3 \in \R$. The quantities $a_{max} \ge
a_{med} \ge a_{min}$ can be conveniently defined to be the maximum,
median and minimum values of $a_1,a_2,a_3$ respectively.

For $f \in \Sch ' $ we denote by $\hat{f}$ or $\ft (f)$ the Fourier transform of $f$ with respect to both spatial and time variables,
$$\hat{f}(\xi , \tau)=\int _{\R ^2} e^{-ix\xi}e^{-it\tau}f(x,t)dxdt .$$
Moreover, we use $\ft_x$ and $\ft_t$ to denote the Fourier transform
with respect to space and time variable respectively. Let $\Z_+=\Z
\cap [0,\infty]$. Let $I_{\le0}=\{\xi:|\xi|<3/2\}$, $\widetilde{I
}_{\le0}=\{\xi:|\xi|\le2\}$. For $k \in \Z_{>0}$ let $I_k$ and
$\widetilde{I}_k$ be dyadic intervals with $I_k \subset
\widetilde{I}_k$. More precisely $ I_k=\{\xi:|\xi|\in
[(3/4)\cdot2^k,(3/2)\cdot2^k]  \}, \widetilde{I}_k=\{\xi:|\xi|\in
[2^{k-1},2^{k+1}]  \}$.

Let $\eta_0: \R \to [0,1]$ denote a smooth bump function supported in $ [-2,2]$ and equal to $1$ in $[-1,1]$.
For $k \in \Z $, let $\chi_k(\xi) = \eta_0(\xi/2^k) - \eta_0(\xi/2^{k-1}) $, which is supported in $\{\xi:|\xi| \in [2^{k-1},2^{k+1}]\} $, and
$$\chi_{[k_1,k_2]}=\sum_{k=k_1}^{k_2} \chi_k \quad \mbox{ for any} \ k_1 \le k_2 \in \Z .$$
For simplicity, let $\eta_k=\chi_k$ if $k \ge 1$ and $\eta_k \equiv 0$ if $k \le -1$. Also, for $ k_1 \le k_2 \in \Z $ let
$$\eta_{[k_1,k_2]}=\sum_{k=k_1}^{k_2} \eta_k \ \mbox{and} \ \eta_{\le k_2}=\sum_{k=-\infty}^{k_2} \eta_k .$$
$\{ \chi_k \}_{k \in \Z}$ is the homogeneous decomposition function sequence and $\{ \eta_k \}_{k \in \Z_+}$ is
the inhomogeneous decomposition function sequence to the frequence space. For $k\in \Z$ let $P_k$ denote the
operators on $L^2(\R)$ defined by $\widehat{P_ku}(\xi)=\mathbf{1}_{I_k}(\xi)\hat{u}(\xi)$. By a slight abuse
of notation we also defined the operators $P_k$ on $L^2(\R \times \R)$ by formulas $\ft(P_ku)(\xi,\tau)=\mathbf{1}_{I_k}(\xi)\ft(u)(\xi,\tau)$.
For $l\in \Z$ let
$$P_{\le l}=\sum_{k \le l}P_k, \quad P_{\ge l}=\sum_{k \ge l}P_k. $$
For $k,j \in \Z_+$ let
$$D_{k,j}=\{(\xi,\tau) \in \R \times \R : \xi \in \widetilde{I_k}, \tau - w(\xi) \in \widetilde{I_j}\}, D_{k,\le j}=\cup_{l\le j}D_{k,l}.$$

For $\xi \in \R$, let
$$w(\xi)=-\xi^5, $$
which is the dispersion relation associated to the equation \eqref{eq:5kdv}. For $\phi \in L^2(\R)$, let $W(t)\phi \in C(\R:L^2)$ be the linear solution given by
$$\ft_x[W(t)\phi](\xi,t)=e^{itw(\xi)}\ft_x(\phi)(\xi).$$
We introduce that $X^{s,b}$ norm associated to Eq. \eqref{eq:5kdv} which is given by
$$\norm{u}_{X^{s,b}}=\norm{\lag \tau - w(\xi)\rag^b\lag \xi \rag^s \ft(u)}_{L^2(\R^2)},$$
where $\lag \cdot \rag = (1+|\cdot|)$. The space $X^{s,b}$ turns out
to be very useful in the study of low-regularity theory for the
dispersive equations. These space were used systematically to study
nonlinear dispersive wave problems by Bourgain \cite{Bou} and used
by Kenig, Ponce and Vega \cite{KPV} and Tao \cite{Tao1}. Klainerman
and Machedon \cite{KM} used similar ideas in their study of the
nonlinear wave equation. We denote the space by $X_{T}$ localized to
the interval $[-T,T]$.

For $k \in \Z_+$, we define the weighted $X^{s,\frac 12,1}$-type space $X_k$ for frequency localized functions $f_k$,
\begin{eqnarray}\label{eq:Xk}
X_k=\left\{
\begin{array}{l}
f\in L^2(\R^2): f(\xi,\tau) \mbox{ is supported in }
\widetilde
{I}_k\times\R  \mbox{ ($\widetilde{I}_{\leq 0}\times \R$ if $k=0$)} \nonumber\\
\mbox{ and }\norm{f}_{X_k}:=\sum_{j=0}^\infty
2^{j/2}\beta_{k,j}\norm{\eta_j(\tau-w(\xi))\cdot
f(\xi,\tau)}_{L^2_{\xi,\tau}}<\infty\
\end{array}
\right\}
\end{eqnarray}
where
\begin{eqnarray}\label{beta weight}
\beta_{k,j} =\left \{
\begin{array}{lr}
2^{j/2}, &k=0,\\
1+2^{(j-5k)/8}, &k \ge 1.
\end{array}
\right.
\end{eqnarray}
\begin{remark}
We choose a parameter $\frac 18$ in the weight. In fact, we can choose any parameter from $1/8$ to $3/16 $.
But this change will not affect our result since the weight helps to reduce only the low-high interactions
where the low frequency part has large modulation.
\end{remark}
As in \cite{IKT} at frequency $2^k$ we will use the $X^{s,\frac 12,1}$ structure given by the $X_k$ norm,
uniformly on the $2^{-2k}$ time scale. For $k\in \Z_+$, we define function spaces
\begin{eqnarray*}
&& F_k=\left\{
\begin{array}{l}
f\in L^2(\R^2): \widehat{f}(\xi,\tau) \mbox{ is supported in }
\widetilde{I}_k\times\R \mbox{ ($\wt{I}_{\leq 0}\times \R$ if $k=0$)} \\
\mbox{ and }\norm{f}_{F_k}=\sup\limits_{t_k\in \R}\norm{\ft[f\cdot
\eta_0(2^{2k}(t-t_k))]}_{X_k}<\infty
\end{array}
\right\},
\\
&&N_k=\left\{
\begin{array}{l}
f\in L^2(\R^2): \widehat{f}(\xi,\tau) \mbox{ is supported in }
\widetilde{I}_k\times\R \mbox{ ($\wt{I}_{\leq 0}\times \R$ if $k=0$)} \mbox{ and } \\
\norm{f}_{N_k}=\sup\limits_{t_k\in
\R}\norm{(\tau-\omega(\xi)+i2^{2k})^{-1}\ft[f\cdot
\eta_0(2^{2k}(t-t_k))]}_{X_k}<\infty
\end{array}
\right\}.
\end{eqnarray*}
Since the spaces $F_k$ and $N_k$ are defined on the whole line, we
define then local versions of the spaces in standard ways. For $T\in
(0,1]$ we define the normed spaces
\begin{align*}
F_k(T)=&\{f\in C([-T,T]:L^2): \norm{f}_{F_k(T)}=\inf_{\wt{f}=f
\mbox{ in } \R\times [-T,T]}\norm{\wt f}_{F_k}\},\\
N_k(T)=&\{f\in C([-T,T]:L^2): \norm{f}_{N_k(T)}=\inf_{\wt{f}=f
\mbox{ in } \R\times [-T,T]}\norm{\wt f}_{N_k}\}.
\end{align*}
We assemble these dyadic spaces in a Littlewood-Paley manner. For
$s\geq 0$ and $T\in (0,1]$, we define function spaces solutions and
nonlinear terms:
\begin{eqnarray*}
&&F^{s}(T)=\left\{ u:\
\norm{u}_{F^{s}(T)}^2=\sum_{k=1}^{\infty}2^{2sk}\norm{P_k(u)}_{F_k(T)}^2+\norm{P_{\leq
0}(u)}_{F_0(T)}^2<\infty \right\},
\\
&&N^{s}(T)=\left\{ u:\
\norm{u}_{N^{s}(T)}^2=\sum_{k=1}^{\infty}2^{2sk}\norm{P_k(u)}_{N_k(T)}^2+\norm{P_{\leq
0}(u)}_{N_0(T)}^2<\infty \right\}.
\end{eqnarray*}
We define the dyadic energy space as follows: For $s\geq 0$ and $u\in
C([-T,T]:H^\infty)$
\begin{eqnarray*}
\norm{u}_{E^{s}(T)}^2=\norm{P_{\leq 0}(u(0))}_{L^2}^2+\sum_{k\geq
1}\sup_{t_k\in [-T,T]}2^{2sk}\norm{P_k(u(t_k))}_{L^2}^2.
\end{eqnarray*}

These $l^1$-type $X^{s,b}$ structures were first introduced in
\cite{Tataru} and used in \cite{IK, IKT, Tao2, GW}. The weighted
$l^1$-type $X^{s,b}$ structures in here were used in \cite{GPWW}.
\begin{lemma}[Properties of $X_k$]
Let $k, l\in \Z_+$ with $l \le 5k$ and $f_k\in X_k$. Then
\begin{equation}\label{eq:pre2}
\begin{split}
&\sum_{j=l+1}^\infty
2^{j/2}\beta_{k,j}\normo{\eta_j(\tau-\omega(\xi))
\int_{\R}|f_k(\xi,\tau')|
2^{-l}(1+2^{-l}|\tau-\tau'|)^{-4}d\tau'}_{L^2}\\
&+2^{l/2}\normo{\eta_{\leq l}(\tau-\omega(\xi)) \int_{\R}
|f_k(\xi,\tau')| 2^{-l}(1+2^{-l}|\tau-\tau'|)^{-4}d\tau'}_{L^2}\les
\norm{f_k}_{X_k}.
\end{split}
\end{equation}
In particular, if $t_0\in \R$ and $\gamma\in \Sch(\R)$, then
\begin{eqnarray}\label{eq:pXk3}
\norm{\ft[\gamma(2^l(t-t_0))\cdot \ft^{-1}(f_k)]}_{X_k}\les
\norm{f_k}_{X_k}.
\end{eqnarray}
\end{lemma}
\begin{proof}
It follows directly from the definition that
\begin{equation}\label{eq:pre1}
\left\|\int_{\R}|f_{k}(\xi,\tau')|d\tau'\right\|_{L_{\xi}^2} \lesssim \norm{f_k}_{X_k}.
\end{equation}
First, assume $k \ge 1$. For the second term on the left-hand side
of \eqref{eq:pre2}, it follows from Cauchy-Schwarz inequality and
\eqref{eq:pre1} that
\begin{align*}
2^{l/2}&\Big\|\eta_{\le l}(\tau-w(\xi))\int_{\R}|f_k(\xi, \tau')|2^{-l}(1+2^{-l}|\tau - \tau'|)^{-4} d\tau'\Big\|_{L^2} \\
&\lesssim \left\|\int_{\R}|f_{k}(\xi,\tau')|d\tau'\right\|_{L_{\xi}^2} \lesssim \norm{f_k}_{X_k}.
\end{align*}
It remains to control the first term on the left-hand side of \eqref{eq:pre2}, let $f_k(\xi,\tau') = \sum_{j_1\ge0}f_{k,j_1}$, where $f_{k,j_1} = f_k(\xi,\tau')\eta_{j_1}(\tau' - w(\xi))$. For $l < j \le 5k$, we have $\beta_{k,j} \sim 1$. Thus, we have from Cauchy-Schwarz inequality that
\begin{align*}
\sum_{l < j \le 5k}\sum_{j_1 \ge 0}&2^{j/2}\Big\|\eta_j(\tau-w(\xi))\int_{\R}f_{k,j_1}(\xi, \tau')2^{-l}(1+2^{-l}|\tau - \tau'|)^{-4} d\tau'\Big\|_{L^2} \\
&\lesssim \sum_{l < j \le 5k}\sum_{j_1 \ge 0}2^{j}2^{3l - 4\max(j,j_1)}\Big\|\int_{\R}f_{k,j_1}(\xi, \tau')d\tau'\Big\|_{L_{\xi}^2} \lesssim \norm{f_k}_{X_k}.
\end{align*}
For the rest term $(j > 5k)$, since $l \le 5k$, we get similarly as before that
\begin{align*}
\sum_{5k < j }\sum_{j_1 \ge 0}&2^{j/2}\beta_{k,j}\Big\|\eta_j(\tau-w(\xi))\int_{\R}f_{k,j_1}(\xi, \tau')2^{-l}(1+2^{-l}|\tau - \tau'|)^{-4} d\tau'\Big\|_{L^2} \\
&\lesssim \sum_{5k < j }\sum_{j_1 \ge 0}2^{j}2^{(j-5k)/8}2^{3l - 4\max(j,j_1)}\Big\|\int_{\R}f_{k,j_1}(\xi, \tau')d\tau'\Big\|_{L_{\xi}^2} \lesssim \norm{f_k}_{X_k}.
\end{align*}
Now consider $k=0$. In this case, we also get the condition $l=0$.
So, it is relatively simpler than $k \ge 1 $ case. Similarly as
before, we have
\begin{align*}
\sum_{j \ge 0}\sum_{j_1 \ge 0}&2^{j/2}\beta_{0,j}\Big\|\eta_j(\tau-w(\xi))\int_{\R}f_{0,j_1}(\xi, \tau')(1+|\tau - \tau'|)^{-4} d\tau'\Big\|_{L^2} \\
&\lesssim \sum_{j \ge 0 }\sum_{j_1 \ge 0}2^{3j/2}2^{- 4\max(j,j_1)}\Big\|\int_{\R}f_{0,j_1}(\xi, \tau')d\tau'\Big\|_{L_{\xi}^2} \lesssim \norm{f_0}_{X_0}.
\end{align*}
Thus, we complete the proof of the lemma.
\end{proof}

As in \cite{IKT}, for any $k\in \Z_+$ we define the set $S_k$ of
$k-acceptable$ time multiplication factors
$$S_k=\{m_k:\R\rightarrow \R: \norm{m_k}_{S_k}=\sum_{j=0}^{10}
2^{-2jk}\norm{\partial^jm_k}_{L^\infty}< \infty\}.$$ Direct
estimates using the definitions and \eqref{eq:pXk3} show that for
any $s\geq 0$ and $T\in (0,1]$
\begin{equation}\label{eq:Sk}
\begin{cases}
\normb{\sum\limits_{k\in \Z_+} m_k(t)\cdot P_k(u)}_{F^{s}(T)}\lesssim (\sup_{k\in \Z_+}\norm{m_k}_{S_k})\cdot \norm{u}_{F^{s}(T)};\\
\normb{\sum\limits_{k\in \Z_+} m_k(t)\cdot
P_k(u)}_{N^{s}(T)}\lesssim
(\sup_{k\in \Z_+}\norm{m_k}_{S_k})\cdot \norm{u}_{N^{s}(T)};\\
\normb{\sum\limits_{k\in \Z_+} m_k(t)\cdot
P_k(u)}_{E^{s}(T)}\lesssim (\sup_{k\in \Z_+}\norm{m_k}_{S_k})\cdot
\norm{u}_{E^{s}(T)}.
\end{cases}
\end{equation}

\begin{remark}\label{re:counterexample}
As mentioned before, the bilinear estimates do not hold on the standard $X^{s,b}$ spaces:
\begin{equation}\label{eq:counterexample}
\norm{\px^3(uv)}_{X^{s,b-1}} \nleq C\norm{u}_{X^{s,b}}\norm{v}_{X^{s,b}},
\end{equation}
due to strong high-low interactions. Indeed, for fixed large frequency $N$, define the characteristic functions supported on the sets:
$$A=\set{(\tau, \xi) \in \R^2 : |\tau - \xi^5| \le 1, N \le |\xi| \le N + 1} \quad \mbox{and} \quad B=\set{(\tau, \xi) \in \R^2 : |\tau - \xi^5| \le 1, |\xi| \le 1},$$
$$\ft(u)(\tau, \xi) = \mathbf{1}_A(\tau, \xi) \quad \mbox{and} \quad \ft(v)(\tau, \xi) = \mathbf{1}_B(\tau, \xi).$$
By simple calculation, we have $\text{LHS} = NN^s$, while $\text{RHS}=N^s$, respectively. However, if one use the short time $X^{s,b}$ spaces defined above, this low-high interaction counter-example is resolved. Concretely, consider the following sets;
$$\wt{A}=\set{(\tau, \xi) \in \R^2 : |\tau - \xi^5| \le N^2, N \le |\xi| \le N + N^{-2}}$$
and
$$\wt{B}=\set{(\tau, \xi) \in \R^2 : |\tau - \xi^5| \le 1, |\xi| \le 1}.$$
And define functions $\wt{u}$ and $\wt{v}$ which satisfy
$$\ft(\widetilde{u})(\tau, \xi) = \mathbf{1}_{\wt{A}}(\tau, \xi) \quad \mbox{and} \quad \ft(\widetilde{v})(\tau, \xi) = \mathbf{1}_{\wt{B}}(\tau, \xi).$$
Computing both side, we have that for any $s \in \R$,
$$\norm{\px^3(\wt{u}\wt{v})}_{N^s} \sim N^sN^3N^{-1}N^{-2}N \sim N^sN \quad \mbox{and} \quad \norm{\wt{u}}_{F^s}\norm{\wt{v}}_{F^s} \sim N^sN.$$
Moreover, this example explains how we choose time length ($=\text{(frequency)}^{-2}$) on which we apply $X^{s,b}$ structures.
\end{remark}

\section{Bilinear estimates}\label{sec:nonlinear}

In this section we show the bilinear estimates. For $\xi_1,\xi_2 \in
\R$, let
$$ H(\xi_1,\xi_2) = \xi_1^5 + \xi_2^5 -(\xi_1+\xi_2)^5 $$
be the resonance function, which plays an crucial role in the
bilinear $X^{s,b}$-type estimates. For compactly supported functions
$f,g,h \in L^2(\R^2)$, we define
$$J(f,g,h) = \int_{\R^4} f(\xi_1,\tau_1)g(\xi_2,\tau_2)h(\xi_1+\xi_2, \tau_1+\tau_2+H(\xi_1,\xi_2)) d\xi_1d\xi_2d\tau_1\tau_2. $$
By simple changes of variables in the integration, we have
$$|J(f,g,h)|=|J(g,f,h)|=|J(f,h,g)|=|J(\widetilde{f},g,h)|,$$
where $\tilde{f}(\xi,\tau)=f(-\xi,-\tau)$.

\begin{lemma}\label{lemma:block estimate} Let $k_i \in \Z,j_i\in \Z_+,i=1,2,3$. Let $f_{k_i,j_i} \in L^2(\R\times\R) $ be nonnegative functions supported in $[2^{k_i-1},2^{k_i+1}]\times \widetilde{I}_{j_i}$.

(a) For any $k_1,k_2,k_3 \in \Z$ with $|k_{max}-k_{min}| \le 5$ and
$j_1,j_2,j_3 \in \Z_+$, then we have
\begin{eqnarray}\label{eq:block estimate1}
   J(f_{k_1,j_1},f_{k_2,j_2},f_{k_3,j_3}) \lesssim 2^{j_{min}/2}2^{j_{med}/4}2^{- \frac34 k_{max}}\prod_{i=1}^3 \|f_{k_i,j_i}\|_{L^2}.
\end{eqnarray}

(b) If $2^{k_{min}} \ll 2^{k_{med}} \sim 2^{k_{max}}$, then for all
$i=1,2,3$ we have
\begin{eqnarray}\label{eq:block estimate2}
 J(f_{k_1,j_1},f_{k_2,j_2},f_{k_3,j_3}) \lesssim 2^{(j_1+j_2+j_3)/2}2^{-3k_{max}/2}2^{-(k_i+j_i)/2}\prod_{i=1}^3 \|f_{k_i,j_i}\|_{L^2}.
\end{eqnarray}

(c) For any $k_1,k_2,k_3 \in \Z$ and $j_1,j_2,j_3 \in \Z_+$, then we
have
\begin{eqnarray}\label{eq:block estimate3}
J(f_{k_1,j_1},f_{k_2,j_2},f_{k_3,j_3}) \lesssim
2^{j_{min}/2}2^{k_{min}/2}\prod_{i=1}^3 \|f_{k_i,j_i}\|_{L^2}.
\end{eqnarray}
\end{lemma}

Lemma \ref{lemma:block estimate} is obtained in a similar way to
Tao's (\cite{Tao1}, Proposition 6.1) in the context of the KdV
equation. For the fifth-order equation, it was first shown by Chen,
Li, Miao and Wu \cite{CLMW}. But there was error in the
\emph{high-high $\to$ high} case and was corrected in \cite{CG}. See
\cite{CG} for the proof. We rewrite the lemma in the following form:

\begin{corollary}\label{cor:block estimate}
Assume $k_i \in \Z,j_i\in \Z_+,i=1,2,3$ and $f_{k_i,j_i} \in
L^2(\R\times\R) $ be functions supported in $D_{k_i,j_i}$, $i=1,2$.

(a) For any $k_1,k_2,k_3 \in \Z$ with $|k_{max}-k_{min}| \le 5$ and
$j_1,j_2,j_3 \in \Z_+$, then we have
\begin{eqnarray}\label{eq:block estimate1.1}
   \| \mathbf{1}_{D_{k_3,j_3}}(\xi,\tau)(f_{k_1,j_1}\ast f_{k_2,j_2})\|_{L^2} \lesssim 2^{j_{min}/2}2^{j_{med}/4}2^{- \frac34 k_{max}}\prod_{i=1}^2 \|f_{k_i,j_i}\|_{L^2}.
\end{eqnarray}

(b) If $2^{k_{min}} \ll 2^{k_{med}} \sim 2^{k_{max}}$, then for all
$i=1,2,3$ we have
\begin{eqnarray}\label{eq:block estimate2.1}
 \| \mathbf{1}_{D_{k_3,j_3}}(\xi,\tau)(f_{k_1,j_1}\ast f_{k_2,j_2})\|_{L^2} \lesssim 2^{(j_1+j_2+j_3)/2}2^{-3k_{max}/2}2^{-(k_i+j_i)/2}\prod_{i=1}^2 \|f_{k_i,j_i}\|_{L^2}.
\end{eqnarray}

(c) For any $k_1,k_2,k_3 \in \Z$ and $j_1,j_2,j_3 \in \Z_+$, then we
have
\begin{eqnarray}\label{eq:block estimate3.1}
\| \mathbf{1}_{D_{k_3,j_3}}(\xi,\tau)(f_{k_1,j_1}\ast
f_{k_2,j_2})\|_{L^2} \lesssim
2^{j_{min}/2}2^{k_{min}/2}\prod_{i=1}^2 \|f_{k_i,j_i}\|_{L^2}.
\end{eqnarray}
\end{corollary}

\begin{remark}
If the assumption is replaced by $f_{k_i,j}$ is supported in
$\widetilde{I}_{k_i} \times \widetilde{I}_{\le j}$ ($D_{k_i,\le j}$)
for $k_1,k_2,k_3 \in \Z_+$, then part (a) and (c) of Lemma
\ref{lemma:block estimate} (and of Corollary \ref{cor:block
estimate}) also hold. In addition, if $k_{min}\ne 0$, then part (b)
holds, else if $k_{min}=0$, part (b) holds for $i\in \{1,2,3\}$ with
$k_i\ne 0$. See \cite{Guo2} for the proof.

\end{remark}

\begin{proposition}[high-low $\Rightarrow$ high]\label{prop:high-low-high}
  Let $k_3 \ge 20$, $|k_2-k_3| \le 4$, $0\le k_1 \le k_2-10$, then we have
  \begin{equation}\label{eq:high-low}
    \|P_{k_3}\px^3(u_{k_1}v_{k_2}) \|_{N_{k_3}} + \|P_{k_3}(u_{k_1}\px^3v_{k_2}) \|_{N_{k_3}} \lesssim \|u_{k_1}\|_{F_{k_1}}\|v_{k_2}\|_{F_{k_2}}.
  \end{equation}
\end{proposition}

\begin{proof}
First, we observe that each term of the left-hand side of
\eqref{eq:high-low} has the same bound since $2^{k_2} \sim 2^{k_3}$.
From the definition, the left-hand side of \eqref{eq:high-low} is
bounded by
\begin{equation}\label{eq:high-low1}
\sup_{t_{k} \in \R} \|(\tau-w(\xi)+ i 2^{2k_3})^{-1} 2^{3k_3}
\mathbf{1}_{I_{k_3}}(\xi)\mathcal{F}[u_{k_1}\eta_0(2^{2k_3-2}(t-t_k))]\ast
\mathcal{F}[v_{k_2}\eta_0(2^{2k_3-2}(t-t_k)) ] \|_{X_{k_3}}.
\end{equation}
Set $f_{k_1}= \mathcal{F}[u_{k_1}\eta_0(2^{2k_3-2}(t-t_k))]$ and
$f_{k_2} = \mathcal{F}[v_{k_2}\eta_0(2^{2k_3-2}(t-t_k)) ]$. We
decompose $f_{k_i}$ into modulation dyadic pieces as $
f_{k_i,j_i}(\xi,\tau) = f_{k_i}(\xi,\tau)\eta_{j_i}(\tau-\xi^5) $,
$i=1,2$, with usual modification $ f_{k_i,\le j}(\xi,\tau) =
f_{k_i}(\xi,\tau)\eta_{\le j}(\tau-\xi^5)$. Then
\eqref{eq:high-low1} is bounded and reduced by
\begin{equation}\label{eq:reduction}
\sup_{{t_{k}}\in \R} 2^{3k_3}(\sum_{0\leq j_3\leq
5k_3}+\sum_{j_3\geq 5k_3}) \frac{2^{j_3
/2}\beta_{k_3,j_3}}{\max(2^{j_3},2^{2k_3})} \sum_{j_1,j_2 \ge 2k_3}
\|\textbf{1}_{D_{k_3,j_3}}\cdot (f_{k_1,j_1}\ast f_{k_2,j_2})
\|_{L^2}:=I+II.
\end{equation}
For the term $I$, by Corollary \ref{cor:block estimate} (b) we get
\begin{align*}
I\les&\sup_{{t_{k}}\in \R} 2^{3k_3}\sum_{0\leq j_3\leq 5k_3} 2^{j_3
/2}2^{-\max(j_3,2k_3)}\sum_{j_1,j_2 \ge 2k_3}
2^{j_1/2}2^{j_2/2}2^{-2k_3}\|f_{k_1,j_1}\|_{L^2}\|f_{k_2,j_2}\|_{L^2}\\
\les&\sup_{{t_{k}}\in \R} \sum_{j_1,j_2 \ge 2k_3}
2^{j_1/2}2^{j_2/2}\|f_{k_1,j_1}\|_{L^2}\|f_{k_2,j_2}\|_{L^2}\les\|u_{k_1}\|_{F_{k_1}}\|v_{k_2}\|_{F_{k_2}}.
\end{align*}
The last inequality comes from the definition of $X_k$-norm and (2.4) in \cite{Guo2}. More precisely,
\begin{align*}
\sum_{j_1 \ge 2k_3}2^{j_1/2}\norm{f_{k_1,j_1}}_{L^2} &\lesssim \sum_{j_1 > 2k_3} 2^{j_1/2}\Big\|\eta_{j_1}(\tau-\omega(\xi))
\int_{\R}|\bar{f}_{k_1}(\xi,\tau')|\cdot
2^{-2k_3}(1+2^{-2k_3}|\tau-\tau'|)^{-4}d\tau'\Big\|_{L^2}\\
&+2^{(2k_3)/2}\normo{\eta_{\leq 2k_3}(\tau-\omega(\xi)) \int_{\R}
|\bar{f}_{k_1}(\xi,\tau')| 2^{-2k_3}(1+2^{-2k_3}|\tau-\tau'|)^{-4}d\tau'}_{L^2}\\
&\lesssim \norm{\ft[u_{k_1}\eta_0(2^{2k_1}(t-t_k))]}_{X_{k_1}} \\
&\lesssim \norm{u_{k_1}}_{F_{k_1}},
\end{align*}
where $\bar{f}_{k_1}=\ft[u_{k_1}\eta_0(2^{2k_1}(t-t_k))]$, and by same method, we have
$$\sum_{j_2 \ge 2k_3}2^{j_2/2}\norm{f_{k_2,j_2}}_{L^2} \lesssim \norm{v_{k_2}}_{F_{k_2}}.$$
This computation shows why we only consider the reduction form \eqref{eq:reduction} and will be used repeatedly in this and next sections. \\
For the term $II$, by the support properties, we have
$\textbf{1}_{D_{k_3,j_3}}\cdot (f_{k_1,j_1}\ast f_{k_2,j_2})\equiv
0$ unless
\begin{align}\label{eq:support}
2^{j_{max}}\sim \max(2^{j_{med}},|H|),
\end{align}
where $|H| \sim |\xi_{max}|^4|\xi_{min}|$ and $|\xi_{max}| = \max(|\xi_1|,|\xi_2|,|\xi_1+\xi_2|), |\xi_{min}|=\min(|\xi_1|,|\xi_2|,|\xi_1+\xi_2|)$.\\
Since in this case $|H|\ll 2^{5k_3}$, then $\max(j_1,j_2)\geq
5k_3-4$. By Corollary \ref{cor:block estimate} (c) we get
\begin{align*}
II\les&\sup_{{t_{k}}\in \R} 2^{3k_3}\sum_{j_3\geq 5k_3} 2^{j_3
/2}\beta_{k_3,j_3}2^{-j_3}\sum_{j_1,j_2 \ge 2k_3}
2^{j_1/2}2^{j_2/2}2^{-5k_3/2}2^{k_1/2}\|f_{k_1,j_1}\|_{L^2}\|f_{k_2,j_2}\|_{L^2}\\
\les&\sup_{{t_{k}}\in \R} \sum_{j_1,j_2 \ge 2k_3}
2^{j_1/2}2^{j_2/2}\|f_{k_1,j_1}\|_{L^2}\|f_{k_2,j_2}\|_{L^2}\les\|u_{k_1}\|_{F_{k_1}}\|v_{k_2}\|_{F_{k_2}}.
\end{align*}
Thus we finish the proof.
\end{proof}

\begin{proposition}[high-high $\Rightarrow$ high]\label{prop:high-high-high}
  Let $k_3 \ge 20$, $|k_1-k_2|, |k_2-k_3| \le 4$, then we have
  \begin{equation}\label{eq:high-high-high}
    \|P_{k_3}\px^3(u_{k_1}v_{k_2}) \|_{N_{k_3}} + \|P_{k_3}(\px^3u_{k_1}v_{k_2}) \|_{N_{k_3}}\lesssim
    2^{-\frac34k_3} \|u_{k_1}\|_{F_{k_1}}\|v_{k_2}\|_{F_{k_2}}.
  \end{equation}
\end{proposition}

\begin{proof}
 As in the proof of Proposition \ref{prop:high-low-high}, both terms of the left-hand side of \eqref{eq:high-high-high} are bounded by
\begin{align}\label{eq:high-high-high1}
\sup_{{t_{k}}\in \R} 2^{3k_3} \sum_{j_3\ge 0} 2^{j_3
/2}\beta_{k_3,j_3} \sum_{j_1,j_2 \ge 0} \|(2^{j_3} + i
2^{2k_3})^{-1} \textbf{1}_{D_{k_3,j_3}}\cdot (f_{k_1,j_1}\ast
f_{k_2,j_2}) \|_{L^2}.
\end{align}
We may assume $2^{j_{max}}\sim 2^{5k_3}$, otherwise it is easier to
handle in view of \eqref{eq:support} and $|H|\sim 2^{5k_3}$. Then by
Corollary \ref{cor:block estimate} (a) and \eqref{eq:pre2} we get
\begin{align*}
\eqref{eq:high-high-high1}&\les \sup_{{t_{k}}\in \R}
2^{3k_3}\sum_{|j_3-5k_3| \le 5}\sum_{j_1,j_2\ge 2k_3} 2^{j_3/2}
2^{-\max(j_3,2k_3)}2^{j_{min}/2}2^{j_{med}/4}2^{-3k_3/4}\|f_{k_1,j_1}\|_{L^2}\|f_{k_2,j_2}\|_{L^2}\\
&\les 2^{-3k_3/4}\|u_{k_1}\|_{F_{k_1}}\|v_{k_2}\|_{F_{k_2}},
\end{align*}
and hence finish the proof of the proposition.
\end{proof}

\begin{proposition}[high-high $\Rightarrow$ low]\label{prop:high-high-low}
(a) Let $k_2 \ge 20$, $|k_1-k_2| \le 4$, $1\le k_3 \le k_2-10$, then
we have
  \begin{equation}\label{eq:high-high-low1}
  \begin{split}
  \|P_{k_3}\px(\px^2u_{k_1}v_{k_2}) \|_{N_{k_3}} &+ \|P_{k_3}\px(\px u_{k_1} \px v_{k_2}) \|_{N_{k_3}} \\
     &\lesssim 2^{(k_2-k_3)/2}(2^{k_2/2-2k_3} + 2^{-\frac32k_3})\|u_{k_1}\|_{F_{k_1}}\|v_{k_2}\|_{F_{k_2}}.
  \end{split}
  \end{equation}
(b) Let $k_2 \ge 20$, $|k_1-k_2| \le 4$, $k_3=0$, then we have
\begin{equation}\label{eq:high-high-low2}
\|P_{k_3}\px(\px^2u_{k_1}v_{k_2}) \|_{N_{k_3}} + \|P_{k_3}\px(\px
u_{k_1} \px v_{k_2}) \|_{N_{k_3}}\les
2^{2k_1}\|u_{k_1}\|_{F_{k_1}}\|v_{k_2}\|_{F_{k_2}}.
\end{equation}
\end{proposition}
\begin{proof}
(a) Since $k_3 \le k_2-10$, we observe that the first term is dominated by other terms. By the definition of $X_{k_3}$ and $N_{k_3}$ one take $X^{0,\frac 12,1}$-structure on time intervals of length $2^{-2k_3}$, while $\|u_{k_1}\|_{F_{k_1}}$ and $\norm{v_{k_2}}_{F_{k_2}}$ are taken in smaller intervals of length $2^{-2k_2}$. We make a partition of intervals of length $2^{-2k_3}$. Let $\gamma: \R \to [0,1]$ denote a smooth function supported in $[-1,1]$ with $ \sum_{m\in \Z} \gamma^2(x-m) \equiv 1$.
  The left-hand side of \eqref{eq:high-high-low1} is dominated by
\begin{equation}\label{eq:high-high-low3}
\begin{split}
C\sup_{t_k\in \R}&2^{k_3}2^{2k_2}\Big\|(\tau-w(\xi) +i 2^{2k_3})^{-1}\mathbf{1}_{I_{k_3}}\\
&\cdot  \sum_{|m| \le C 2^{2k_2-2k_3}} \mathcal{F}[u_{k_1}\eta_0(2^{2k_3}(t-t_k))\gamma (2^{2k_2}(t-t_k)-m)]\\
 &\ast \mathcal{F}[u_{k_1}\eta_0(2^{2k_3}(t-t_k))\gamma (2^{2k_2}(t-t_k)-m)]\Big\|_{X_{k_3}}.
\end{split}
\end{equation}
As in the proof of Proposition \ref{prop:high-low-high}, \eqref{eq:high-high-low3} dominated by
\begin{equation}\label{eq:high-high-low4}
2^{k_3}2^{2k_2}\sum_{j_3 \ge 2k_3} 2^{j_3/2}\beta_{k_3,j_3}2^{2k_2-2k_3}\sum_{j_1,j_2 \ge 2k_2} \|(\tau-w(\xi) + i2^{2k_3})^{-1} \textbf{1}_{D_{k_3,j_3}}\cdot (f_{k_1,j_1}\ast f_{k_2,j_2}) \|_{L^2}
\end{equation}
By Corollary \ref{cor:block estimate} (b), we have
\begin{equation*}
2^{4k_2-k_3}\sum_{j_3 \ge 2k_3}\sum_{j_1,j_2 \ge 2k_2}2^{-j_3/2}\beta_{k_3,j_3}2^{(j_1+j_2+j_3)/2} 2^{-3k_2/2}2^{-(k_i+j_i)/2}\norm{f_{k_1,j_1}}_{L^2} \norm{f_{k_2,j_2}}_{L^2}.
\end{equation*}
If $2^{j_3} \sim 2^{j_{max}}$, take $j_i=j_3$. Then by the support property \eqref{eq:support}, we get
\begin{align*}
\eqref{eq:high-high-low3} &\lesssim 2^{4k_2-k_3}\sum_{j_1,j_2 \ge 2k_2}2^{-\frac12(4k_2+k_3)}2^{(k_2-k_3)/2}2^{(j_1+j_2)/2} 2^{-3k_2/2}2^{-k_3/2}\norm{f_{k_1,j_1}}_{L^2} \norm{f_{k_2,j_2}}_{L^2}\\
&\lesssim 2^{k_2-\frac52k_3}\sum_{j_1,j_2 \ge 2k_2}2^{(j_1+j_2)/2} \norm{f_{k_1,j_1}}_{L^2} \norm{f_{k_2,j_2}}_{L^2}.
\end{align*}
Otherwise, since $j_3 \le 4k_2+k_3-10$, \eqref{eq:high-high-low4} can be rewritten
$$2^{4k_2-k_3}\sum_{2k_3 \le j_3 \le 4k_2+k_3 -10 }\sum_{j_1,j_2 \ge 2k_2}2^{-j_3/2}\beta_{k_3,j_3}2^{(j_1+j_2+j_3)/2} 2^{-3k_2/2}2^{-(k_i+j_i)/2}\norm{f_{k_1,j_1}}_{L^2} \norm{f_{k_2,j_2}}_{L^2}.$$
Take $j_i=j_{max}$. Then
\begin{align*}
\eqref{eq:high-high-low3} &\lesssim 2^{4k_2-k_3}\sum_{j_1,j_2 \ge 2k_2}2^{(k_2-k_3)/2}2^{(j_1+j_2)/2} 2^{-3k_2/2}2^{-k_2/2}2^{-j_{max}/2}\norm{f_{k_1,j_1}}_{L^2} \norm{f_{k_2,j_2}}_{L^2}\\
&\lesssim 2^{(k_2-k_3)/2}2^{-\frac32k_3}\sum_{j_1,j_2 \ge 2k_2}2^{(j_1+j_2)/2} \norm{f_{k_1,j_1}}_{L^2} \norm{f_{k_2,j_2}}_{L^2}.
\end{align*}

\noindent(b) The left-hand side of \eqref{eq:high-high-low2} is dominated by
\begin{align}\label{eq:high-high-lowp1}
C\sup_{t_k\in \R}&2^{2k_2}\Big\|\xi(\tau-w(\xi) +i
)^{-1}\mathbf{1}_{I_{0}}\cdot
\mathcal{F}[u_{k_1}\eta_0(t-t_k)]\ast
\mathcal{F}[v_{k_2}\eta_0(t-t_k)]\Big\|_{X_{0}}.
\end{align}
We decompose further the low frequency, then
\begin{align*}\label{eq:high-high-lowp2}
\eqref{eq:high-high-lowp1}\les \sum_{l\leq 0}\sup_{t_k\in
\R}2^l&\left(\sum_{j_3 \le 4k_2 +l - 5} + \sum_{j_3 \ge 4k_2 +l +5} + \sum_{|j_3-4k_2-l| <5} \right)2^{2k_2}2^{-j_3}\\
&\Big\|1_{D_{l,j_3}}\mathcal{F}[u_{k_1}\eta_0(t-t_k)]\ast
\mathcal{F}[v_{k_2}\eta_0(t-t_k)]\Big\|_{X_{0}} = I+II+III.
\end{align*}
From the support property \eqref{eq:support}, the first case includes $j_3 = j_{min}$ and $j_3 = j_{med}$ with $2^{j_{med}} \ll 2^{j_{max}} \sim |H|$ cases, and we regard the second one as $2^{j_3} = 2^{j_{max}} \sim 2^{j_{med}} \gtrsim |H|$ case. The last term is regarded as $j_3=j_{max}$ with $2^{j_3} \sim |H|$ case.  \\
For $I,II$ cases, we use same argument to \eqref{eq:high-high-low3};
$$I \lesssim 2^{4k_2}\sum_{l \le 0}2^l\sum_{j_3 \le 4k_2+l-5} 2^{-j_3/2}\beta_{0,j_3}\sum_{j_1,j_2 \ge 2k_2} \|\textbf{1}_{D_{l,j_3}}\cdot (f_{k_1,j_1}\ast f_{k_2,j_2}) \|_{L^2}, $$
and
$$II \lesssim 2^{4k_2}\sum_{l \le 0}2^l\sum_{j_3 \ge 4k_2+l+5} 2^{-j_3/2}\beta_{0,j_3}\sum_{j_1,j_2 \ge 2k_2} \|\textbf{1}_{D_{l,j_3}}\cdot (f_{k_1,j_1}\ast f_{k_2,j_2}) \|_{L^2}.$$
\noindent First, we consider $I$.\\
From Corollary \ref{cor:block estimate} (b), we estimate that
$$I \lesssim \sum_{l \le 0}2^{l}2^{4k_2}\sum_{j_3 \le 4k_2+l-5}\sum_{j_1,j_2 \ge 2k_2}2^{-j_3/2}\beta_{0,j_3}2^{(j_1+j_2+j_3)/2} 2^{-3k_2/2}2^{-(k_i+j_i)/2}\norm{f_{k_1,j_1}}_{L^2} \norm{f_{k_2,j_2}}_{L^2}.$$
Since $\beta_{0,j_3} = 2^{j_3/2}$, by taking $j_i=j_{max}$ and performing $j_3$ summation, we have
\begin{align*}
I &\lesssim 2^{4k_2}\sum_{l \le 0}2^l\sum_{j_3 \le 4k_2+l-5}\sum_{j_1,j_2 \ge 2k_2} 2^{j_3/2} 2^{(j_1+j_2)/2} 2^{-2k_2} 2^{-j_{max}/2} \norm{f_{k_1,j_1}}_{L^2} \norm{f_{k_2,j_2}}_{L^2} \\
&\lesssim 2^{2k_2}\sum_{l \le 0}2^l2^{-(4k_2+l)/2}\sum_{j_3 \le 4k_2+l-5}\sum_{j_1,j_2 \ge 2k_2} 2^{j_3/2} 2^{(j_1+j_2)/2} \norm{f_{k_1,j_1}}_{L^2} \norm{f_{k_2,j_2}}_{L^2} \\
&\lesssim 2^{2k_2 }\sum_{l \le 0}2^l\sum_{j_1,j_2 \ge 2k_2} 2^{(j_1+j_2)/2} \norm{f_{k_1,j_1}}_{L^2} \norm{f_{k_2,j_2}}_{L^2} \\
&\lesssim 2^{2k_2} \norm{u_{k_1}}_{F_{k_1}}\norm{v_{k_2}}_{F_{k_2}}.
\end{align*}
Now, we consider $II$.\\
In this case, since $2^{j_{max}} \sim 2^{j_{med}}$, we estimate from Corollary \ref{cor:block estimate} (c) that
\begin{align*}
II &\lesssim 2^{4k_2}\sum_{l \le 0}2^l\sum_{j_3 > 4k_2+l+5}\sum_{j_1,j_2 \ge 2k_2} 2^{-j_3/2}\beta_{0,j_3} 2^{j_{min}/2} 2^{l/2}2^{j_{med}/2}2^{-j_{med}/2} \norm{f_{k_1,j_1}}_{L^2} \norm{f_{k_2,j_2}}_{L^2} \\
&\lesssim 2^{4k_2}\sum_{l \le 0}2^l\sum_{j_3 > 4k_2+l+5}\sum_{j_1,j_2 \ge 2k_2} 2^{j_{min}/2} 2^{l/2}2^{j_{med}/2}2^{-j_{med}/2} \norm{f_{k_1,j_1}}_{L^2} \norm{f_{k_2,j_2}}_{L^2} \\
&\lesssim 2^{4k_2}\sum_{l \le 0}2^l\sum_{j_3 > 4k_2+l+5}\sum_{j_1,j_2 \ge 2k_2} 2^{-j_3/2}2^{j_{min}/2} 2^{l/2}2^{j_{med}/2} \norm{f_{k_1,j_1}}_{L^2} \norm{f_{k_2,j_2}}_{L^2} \\
&\lesssim 2^{4k_2 }\sum_{l \le 0}2^l2^{-(4k_2+l)/2}2^{l/2}\sum_{j_1,j_2 \ge 2k_2} 2^{(j_1+j_2)/2} \norm{f_{k_1,j_1}}_{L^2} \norm{f_{k_2,j_2}}_{L^2} \\
&\lesssim 2^{2k_2} \norm{u_{k_1}}_{F_{k_1}}\norm{v_{k_2}}_{F_{k_2}}.
\end{align*}
For the rest term $III$, we get from Plancherel's identity and $X_k$ embedding that
\begin{align*}
\eqref{eq:high-high-lowp1}\les&\sum_{l\leq 0}\sup_{t_k\in
\R}2^l2^{2k_2}\Big\|1_{|\xi|\sim
2^l}\mathcal{F}[u_{k_1}\eta_0(t-t_k)]\ast
\mathcal{F}[v_{k_2}\eta_0(t-t_k)]\Big\|_{L^2}\\
\les&\sup_{t_k\in \R}2^{2k_2}\Big\|u_{k_1}\eta_0(t-t_k)\cdot
v_{k_2}\eta_0(t-t_k)\Big\|_{L_t^2L_x^1}\les
2^{2k_2}\|u_{k_1}\|_{L_t^\infty L_x^2}\|v_{k_2}\|_{L_t^\infty
L_x^2}\\
\les&\sup_{t_{k_1} t_{k_2} \in \R}2^{2k_2}\|u_{k_1}\eta_0(2^{2k_1}(t-t_{k_1}))\|_{L_t^\infty L_x^2}\|v_{k_2}\eta_0(2^{2k_2}(t-t_{k_2}))\|_{L_t^\infty
L_x^2} \\
\les& 2^{2k_2}\|u_{k_1}\|_{F_{k_1}}\|v_{k_2}\|_{F_{k_2}}.
\end{align*}

\end{proof}
\begin{proposition}[low-low $\Rightarrow$ low]\label{prop:low-low-low}
  Let $0 \le k_1,k_2,k_3 \le 200$, then we have
  \begin{equation}\label{eq:low-low-low}
    \|P_{k_3}\px^3(u_{k_1}v_{k_2}) \|_{N_{k_3}} + \|P_{k_3}(\px^3u_{k_1}v_{k_2}) \|_{N_{k_3}}\lesssim \|u_{k_1}\|_{F_{k_1}}\|v_{k_2}\|_{F_{k_2}}
  \end{equation}
\end{proposition}

\begin{proof}
  As in the proof of Proposition \ref{prop:high-high-high}, use \eqref{eq:block estimate1.1}, then we can get \eqref{eq:low-low-low}.
\end{proof}

As a conclusion to this section we prove the bilinear estimates, using the dyadic bilinear estimates obtained above.

\begin{proposition}\label{prop:bilinear}
(a) If $s \ge 1$, $T\in(0,1]$, and $u,v\in F^s(T)$ then
\begin{align}\label{eq:bilinear}
&\|\px(\px u\px v)\|_{N^s(T)} +\|\px(\px^2 u
v)\|_{N^s(T)}+\|\px(u\px^2 v)\|_{N^s(T)}\nonumber\\
&\les \|u\|_{F^{1}(T)}\|v\|_{F^s(T)} +
\|u\|_{F^s(T)}\|v\|_{F^{1}(T)}.
\end{align}
(b) If $T \in (0,1]$, $u \in F^0(T)$ and $v \in F^{2}(T)$, then
\begin{align}\label{eq:bilinear1}
\|\px(\px u\px v)\|_{N^0(T)} +\|\px(\px^2 u
v)\|_{N^0(T)}+\|\px(u\px^2 v)\|_{N^0(T)}\les
\|u\|_{F^0(T)}\|v\|_{F^{2}(T)}.
\end{align}
\end{proposition}

\begin{proof}
The proof follows from the dyadic bilinear estimates and Young's
inequality. See \cite{Guo2} for a similar proof.
\end{proof}
\begin{remark}\label{re:cor1.2-bilinear}
The equation \eqref{eq:5kdv heirarchy} has an additional term $ u^2u_x$. To prove Corollary \ref{cor:gwp}, we need to show the nonlinear estimate for this cubic term. In fact, this term is much easier to handle. In view of the proof of each Lemma, one need to use $L^2$ estimate twice, and it is enough to contorl
$$\sum_{j \ge 2k}2^{j/2}\norm{\textbf{1}_{D_{k,j}}\cdot (f_{k_1,j_1}\ast f_{k_2,j_2})}_{L^2}.$$
Indeed, since there is no derivative, we get from the block estimates that
\begin{equation}\label{eq:cor1.2-bilinear}
\sum_{j \ge 2k}2^{j/2}\norm{\textbf{1}_{D_{k,j}}\cdot (f_{k_1,j_1}\ast f_{k_2,j_2})}_{L^2} \lesssim 2^{2k_1 + \frac12k}\norm{u_1}_{F_{k_1}} \norm{u_2}_{F_{}k_2},
\end{equation}
which implies
$$\norm{\px(u_1u_2u_3)}_{N^2} \lesssim \norm{u_1}_{F^2} \norm{u_2}_{F^2} \norm{u_3}_{F^2}.$$
Moreover, we also get
$$\norm{\px(u^2v)}_{N^0} \le \norm{u}_{N^0}^2\norm{v}_{N^2}.$$
\end{remark}

\section{Energy estimates}\label{sec:energy}

In this section we prove the energy estimates, following the idea in \cite{IKT}. We introduce a new Littlewood-Paley decomposition with smooth cut-offs. With
$$\chi_k(\xi)=\eta_0(\xi/2^k) - \eta_0(\xi/2^{k-1}), \quad k \in \Z.$$
Let $\widetilde{P}_k$ denote the operator on $L^2(\R)$ defined by the Fourier multiplier $\chi_k(\xi)$. Assume that $u,v \in C([-T,T];L^2)$ and
\begin{equation}\label{eq:energy}
\begin{cases}
\pt u + \px^5u = v, \quad  (x,t) \in \R \times (-T,T) \\
u(0,x) =u_0(x)
\end{cases}
\end{equation}
Then we multiply by $u$ and integrate to conclude that
\begin{equation}\label{eq:energy 1}
\sup_{|t_k| \le T}\|u(t_k)\|_{L^2}^2 \le \|u_0\|_{L^2}^2 + \sup_{|t_k| \le T}\Big{|}\int_{\R \times [0,t_k]}u \cdot v dxdt \Big{|}.
\end{equation}

\begin{lemma}\label{lem:energy}
Let $T \in (0,1]$ and $k_1,k_2,k_3 \in \Z_+$.\\
(a) Assume $|k_{max}-k_{min}| \le 5$ and $u_i \in F_{k_i}(T), i=1,2,3$. Then
 \begin{equation}\label{eq:energy 2.1}
\left| \int_{\R\times [0,T]} u_1u_2u_3 dxdt \right| \lesssim 2^{- \frac74k_{max}} \prod_{i=1}^{3} \norm{u_i}_{F_{k_i}(T)}
 \end{equation}
(b) Assume $k_{max} \ge 10$, $2^{k_{min}} \ll 2^{k_{med}} \sim 2^{k_{max}}$ and $u_i \in F_{k_i}(T), i=1,2,3$. Then
 \begin{equation}\label{eq:energy 2.2}
\left| \int_{\R\times [0,T]} u_1u_2u_3 dxdt \right| \lesssim  2^{-2k_{max}-\frac12k_{min}}\prod_{i=1}^{3} \norm{u_i}_{F_{k_i}(T)}
 \end{equation}
(c) Assume $k_1 \le k -10 $. Then
\begin{equation}\label{eq:energy 3}
\left| \int_{\R\times[0,T]} \widetilde{P}_k(u)\widetilde{P}_k(\px^3u\cdot \widetilde{P}_{k_1}v) dxdt    \right| \lesssim 2^{\frac12k_1} \norm{\widetilde{P}_{k_1}v}_{F_{k_1}(T)}\sum_{|k'-k|\le 10} \norm{\widetilde{P}_{k'}(u)}^2_{F_{k'}(T)}
\end{equation}
(d) Under the same condition as in (c), we have
\begin{equation}\label{eq:energy 9}
\left| \int_{\R\times[0,T]} \widetilde{P}_k(u)\widetilde{P}_k(\px^2u\cdot \widetilde{P}_{k_1}\px v) dxdt    \right| \lesssim 2^{\frac12k_1} \norm{\widetilde{P}_{k_1}v}_{F_{k_1}(T)}\sum_{|k'-k|\le 10} \norm{\widetilde{P}_{k'}(u)}^2_{F_{k'}(T)}
\end{equation}
\end{lemma}

\begin{proof}

For (a) and (b), we may assume that $k_1\le k_2\le k_3$ by symmetry.
We fix extensions $\wt{u_i} \in F_{k_i}$ so that
$\norm{\wt{u_i}}_{F_{k_i}} \le 2 \norm{u_i}_{F_{k_i}(T)}, i=1,2,3 $.
Let $\gamma : \R \to [0,1]$ be a smooth partition of unity function
(i.e. $\text{supp } \gamma \subset [-1,1] $ and $\sum_{n\in \Z}
\gamma^3(x-n) \equiv 1, x\in \R $). The left-hand side of
\eqref{eq:energy 2.1} and \eqref{eq:energy 2.2} is bounded by
\begin{equation}\label{eq:energy 4}
          C \sum_{|n| \lesssim 2^{2k_3}} \Big|\int_{\R\times\R} (\gamma(2^{2k_3}t-n)\mathbf{1}_{[0,T]}(t)\wt{u}_1)
        \cdot (\gamma(2^{2k_3}t -n)\wt{u}_2)\cdot (\gamma(2^{2k_3}t-n)\wt{u}_3) dxdt  \Big|
\end{equation}
Set $A=\{n: \gamma(2^{2k_3}t-n) \mathbf{1}_{[0,T]}(t) \mbox{
non-zero and } \neq \gamma(2^{2k_3}t-n) \}$. Then one observe $|A|
\le 4 $.

(a) First consider the summation over $n \in A^c$. let
$f_{k_i}=\ft(\gamma(2^{2k_3}t-n)\widetilde{u}_i)$ and
$f_{k_i,j_i}=\eta_{j_i}(\tau-\xi^5)f_{k_i}$, $i=1,2,3$. By
Parseval's formula and \eqref{eq:pre2}, we have
\begin{equation}\label{eq:energy 5}
\eqref{eq:energy 4}\les \sup_{n \in A^c}2^{2k_3}\sum_{j_1,j_2,j_3
\ge 0}|J(f_{k_1,j_1},f_{k_2,j_2},f_{k_3,j_3})|.
\end{equation}
By \eqref{eq:pre2} and the support properties \eqref{eq:support} we
may assume $j_1,j_2,j_3\geq 2k_3, 2^{j_{max}}\sim 2^{5k_3}$. Then
using Lemma \ref{lemma:block estimate} (a), we get that
\begin{align*}
\eqref{eq:energy 5} &\lesssim \sum_{j_1,j_2,j_3 \ge 2k_3}2^{2k_3}2^{j_{min}/2} 2^{j_{med}/4}2^{-3k_{max}/4}\prod_{i=1}^3\norm{f_{k_i,j_i}}_{L^2}\\
&\les2^{-\frac74k_{max}}\prod_{i=1}^3(2^{2k_i/2}\norm{f_{k_i,j_i}}_{L^2})\les
2^{-\frac74k_{max}}\prod_{i=1}^{3} \norm{u_i}_{F_{k_i}(T)}.
\end{align*}

For the summation over $n\in A$, it is easy to handle since $|A|\leq
4$. Indeed, we observe that if $I\subset \R$ is an interval, $k\in
\Z_+$, $f_k\in X_k$, and $f_k^I=\ft(1_I(t)\cdot \ft^{-1}(f_k))$ then
\begin{align*}
\sup_{j\in \Z_+}2^{j/2}\norm{\eta_j(\tau-\omega(\xi))\cdot
f_k^I}_{L^2}\les \norm{f_k}_{X_k}.
\end{align*}
See \cite{Guo2} for the proof.

(b) This part is a little trickier to prove. Thus, to overcome a trouble, we will use following form, which is different from \eqref{eq:energy 4}, and it suffices to prove that
\begin{align}\label{eq:energyk1=0}
\Big|\sum_{n\in A^c} \int_{\R^2}
(\eta_0(2^{2k_1}t-2^{2k_1-2k_3}n)\wt{u}_1)(\gamma^2(2^{2k_3}t
-n)\wt{u}_2)(\gamma(2^{2k_3}t-n)\wt{u}_3) dxdt \Big|\les
2^{-2k_{2}}\prod_{i=1}^{3} \norm{\tilde u_i}_{F_{k_i}}.
\end{align}
Let $f_1=\eta_0(t-2^{-2k_3}n)\wt{u}_1$ and $f_i=\gamma(2^{2k_3}t
-n)\wt{u}_i$, $i=2,3$.\\
First we consider the case $k_1\geq 1$. Similarly, we get
\eqref{eq:energy 5} and only need to consider the sum over $n\in
A^c$. By Lemma \ref{lemma:block estimate} (b), the left-hand side of \eqref{eq:energyk1=0} is dominated by
\begin{equation}\label{eq:energyk>0}
\sum_{j_1 \ge 2k_1}\sum_{j_2,j_3 \ge 2k_3}2^{2k_3}2^{(j_1+j_2+j_3)/2} 2^{-3k_{max}/2}2^{-(k_i+j_i)/2}\prod_{i=1}^3\norm{f_{k_i,j_i}}_{L^2}.\end{equation}
If $j_1 \neq j_{max}$, take $j_i=j_{max}$. Then since $2^{k_2} \sim 2^{k_3}$, \eqref{eq:support} yields
\begin{equation*}
\eqref{eq:energyk>0} \lesssim 2^{-2k_{max}-\frac12k_{min}}\prod_{i=1}^3\sum_{j_i \ge 2k_i} 2^{j_i/2}\norm{f_{k_i,j_i}}_{L^2} \lesssim \prod_{i=1}^{3}\norm{u_{i}}_{F_{k_i}}.
\end{equation*}
If $j_1=j_{max}$, take $j_i=j_1$. Then from a similar argument, we get
\begin{align*}
\eqref{eq:energyk>0} &\lesssim 2^{-\frac32k_{max}-k_{min}}2^{(k_1-k_3)/2} \sum_{j_1 \ge 4k_3+k_1-5}2^{j_1/2}\beta_{k_1,j_1}\norm{f_{k_1,j_1}}_{L^2} \prod_{i=2}^3\sum_{j_i \ge 2k_3} 2^{j_i/2}\norm{f_{k_i,j_i}}_{L^2}\\
&\lesssim \prod_{i=1}^{3}\norm{u_{i}}_{F_{k_i}},
\end{align*}
where we used $\beta_{k_1,j_1} \gtrsim 2^{(k_3-k_1)/2}$.

\noindent Next we assume $k_1=0$. In this case, we only need to consider the sum over $n \in A^c$. We decompose further the low frequency
component $f_1=\sum_{l\leq 0}f_{1,l}$ with
$f_{1,l}=\ft^{-1}1_{|\xi|\sim 2^l}\ft f_1$. Then
\begin{align*}
\mbox{LHS of }\eqref{eq:energyk1=0}\les& \sum_{l\leq
0}\Big|\sum_{n\in
A^c}\int_{\R\times\R}{f_{1,l}}\cdot {f_2}\cdot {f_3}dxdt\Big|\\
\leq&\sum_{l\leq 0}\Big|\sum_{n\in
A^c}\int_{\R\times\R}{f^H_{1,l}}\cdot {f_2}\cdot
{f_3}dxdt\Big|+\sum_{l\leq 0}\Big|\sum_{n\in
A^c}\int_{\R\times\R}{f^L_{1,l}}\cdot {f_2}\cdot
{f_3}dxdt\Big|=I+II,
\end{align*}
where $\ft (f^L_{1,l})=\eta_{\leq 4k_2+l-30}(\tau-\xi^5)\ft
f_{1,l}$, and $f^H_{1,l}=f_{1,l}-f^L_{1,l}$. The term $II$ is easy
to handle (high frequency with large modulation). Indeed,
\begin{align*}
II\les& 2^{2k_3}\sum_{l\leq 0}2^l\sum_{j_1,j_2,j_3\geq
0}|J(f_{k_1,j_1},f_{k_2,j_2},f_{k_3,j_3})|\\
\les& 2^{2k_3}\sum_{l\leq 0}2^l\sum_{j_1,j_2,j_3\geq
0}2^{(j_1+j_2+j_3)/2}2^{-2k_2}2^{-\max(j_2,j_3)/2} \prod_{i=1}^3(\norm{f_{k_i,j_i}}_{L^2})\les
2^{-2k_{2}}\prod_{i=1}^{3} \norm{\tilde u_i}_{F_{k_i}},
\end{align*}
since $\max(j_2,j_3)\geq l+4k_2-5$ by the support property.

Now we deal with the term $I$. By orthogonality, we may assume
$f_2,f_3$ has frequency support in a ball of size $2^l$.
\[I\leq(\sum_{l\leq -2k_2-1}+\sum_{-2k_2\leq l\leq 0})\Big|\sum_{n\in
A^c}\int_{\R\times\R}{f^H_{1,l}}\cdot {f_2}\cdot
{f_3}dxdt\Big|=I_1+I_2.\]
For the term $I_1$, by H\"older's inequality and using a weight, we have
\begin{align*}
I_2\leq& \sum_{-2k_2\leq l\leq 0}(\sup_{n\in
A^c}\norm{f^H_{1,l}}_{L_{t,x}^2})(\sup_{n\in
A^c}\norm{f_3}_{L_{t,x}^\infty})\sum_{n\in
A^c}\norm{f_2}_{L_{t,x}^2}\\
\leq& \sum_{-2k_2\leq l\leq 0}(\sup_{n\in A^c}\sum_{j_1\geq
l+4k_2-10}\norm{\eta_{j_1}(\tau-\omega(\xi))\ft
f_1}_{L_{\xi,\tau}^2})(2^{l/2}\norm{f_3}_{L_{t}^\infty
L_x^2})2^{k_2}\norm{f_2}_{L_t^\infty L_{x}^2}\\
\les&\sum_{-2k_2\leq l\leq
0}2^{-4k_2-l}2^{l/2}2^{k_2}\prod_{i=1}^{3} \norm{\tilde
u_i}_{F_{k_i}}\les 2^{-2k_2}\prod_{i=1}^{3}\norm{\tilde
u_i}_{F_{k_i}}.
\end{align*}
For the term $I_1$, as $II$, we get from Corollary \ref{cor:block
estimate} (c) and \eqref{eq:pre2} that
\begin{align*}
I_1\les& 2^{2k_3}\sum_{l\leq -2k_2}2^l\sum_{j_1,j_2,j_3\geq
0}|J(f_{k_1,j_1},f_{k_2,j_2},f_{k_3,j_3})|\\
\les& 2^{2k_3}\sum_{l\leq -2k_2}2^l\sum_{j_1,j_2,j_3\geq
0}2^{j_1/2}2^{l/2}\prod_{i=1}^3(\norm{f_{k_i,j_i}}_{L^2})\les
2^{-2k_{2}}\prod_{i=1}^{3} \norm{\tilde u_i}_{F_{k_i}},
\end{align*}
since by \eqref{eq:pre2} we may assume $j_2,j_3\geq 2k_2$.

(c) We denote the commutator of $T_1,T_2$ by
$[T_1,T_2]=T_1T_2-T_2T_1$. Then the left hand side of
\eqref{eq:energy 3} is dominated by
\begin{equation*}
\Big{|}\int_{\R \times [0,T]}\widetilde{P}_k(u)\widetilde{P}_k(\px^3u)\widetilde{P}_{k_1}(v)dxdt \Big{|} + \Big{|}\int_{\R \times [0,T]}\widetilde{P}_k(u)[\widetilde{P}_k,\widetilde{P}_{k_1}(v)](\px^3u)dxdt \Big{|} =: I+II.
\end{equation*}
Using $uu_{xxx}=\frac12(u^2)_{xxx}-3(u_x^2)_x$ and integration by parts, we get
\begin{align*}
I &\lesssim \Big{|}\int_{\R \times [0,T]}\px^3((\widetilde{P}_k(u))^2)\widetilde{P}_{k_1}(v)dxdt \Big{|} + \Big{|}\int_{\R \times [0,T]}\px((\widetilde{P}_k(\px u))^2)\widetilde{P}_{k_1}(v)dxdt \Big{|} \\
&\lesssim \Big{|}\int_{\R \times [0,T]}\widetilde{P}_k(u)\widetilde{P}_k(u)\widetilde{P}_{k_1}(\px^3v)dxdt \Big{|} + \Big{|}\int_{\R \times [0,T]}\widetilde{P}_k(u_x)\widetilde{P}_k(u_x)\widetilde{P}_{k_1}(\px v)dxdt \Big{|} \\
&=:I_1+I_2.
\end{align*}
Then, use \eqref{eq:energy 2.2} to conclude that
\begin{align*}
I_1+I_2 &\lesssim 2^{-2k-\frac12k_1}2^{3k_1}\|\widetilde{P}_k(u) \|_{F_k(T)}^2\|\widetilde{P}_{k_1}(v)\|_{F_{k_1}(T)}\\
&+ 2^{-2k-\frac12k_1}2^{2k}2^{k_1}\|\widetilde{P}_k(u) \|_{F_k(T)}^2\|\widetilde{P}_{k_1}(v)\|_{F_{k_1}(T)},
\end{align*}
which suffices for \eqref{eq:energy 3}.

To control $II$, we use the formula
\begin{equation}\label{eq:energy 8}
\begin{split}
\ft([\widetilde{P}_k,&\widetilde{P}_{k_1}(v)](\px^3u))(\xi,\tau)\\
&=C\int_{\R^2}\ft(\widetilde{P}_{k_1}(\px^3v))(\xi_1,\tau_1) \cdot \ft(u)(\xi-\xi_1,\tau-\tau_1)\cdot m(\xi,\xi_1)d\xi_1d\tau_1,
\end{split}
\end{equation}
where
$$|m(\xi,\xi_1)|=\Big{|}\frac{(\xi-\xi_1)^3(\chi_k(\xi)-\chi_k(\xi-\xi_1))}{\xi_1^3}\Big{|} \lesssim \sum_{|k-k'|\le 4}\chi_{k'}(\xi-\xi_1).$$
By the Parseval's theorem and \eqref{eq:energy 2.2}, we estimate
\begin{align*}
II &\lesssim
2^{-2k-\frac12k_1}2^{3k_1}\|\widetilde{P}_{k_1}(v)\|_{F_{k_1}(T)}\sum_{|k-k'|
\le 4}\|\widetilde{P}_{k'}(u) \|_{F_{k'}(T)}^2, \lesssim \text{RHS
of } \eqref{eq:energy 3}.
\end{align*}

(d) is proved similarly and so we omit the detail.
\end{proof}

\begin{proposition}\label{prop:energy1}
Let $T\in (0,1]$ and $u \in C([-T,T]:H^\infty) $ be a solution to \eqref{eq:5kdv}. Let $s  \ge 5/4$. Then we have
\begin{equation}\label{eq:energy 10}
\norm{u}^2_{E^s(T)} \lesssim \norm{u_0}^2_{H^s} + \|u\|_{F^{5/4}(T)} \norm{u}^2_{F^s(T)}
\end{equation}
\end{proposition}

\begin{proof}
From the definition we have
\begin{equation*}
\norm{u}_{E^s(T)}^2 -\norm{P_{\le 0}(u_0)}_{L^2}^2 \lesssim \sum_{k \ge 1}\sup_{t_k \in [-T,T]}2^{2sk}\norm{\widetilde{P}_k(u(t_k))}_{L^2}^2.
\end{equation*}
Then we get from \eqref{eq:energy 1}
\begin{align*}
2^{2sk}\norm{\widetilde{P}_k(u(t_k))}_{L^2}^2
-2^{2sk}\norm{\widetilde{P}_k(u_0)}_{L^2}^2
&\lesssim 2^{2sk}\Big{|}\int_{\R \times [0,t_k]}\widetilde{P}_k(u)\widetilde{P}_k(u \cdot \px^3u)dxdt\Big{|}\\
&+2^{2sk}\Big{|}\int_{\R \times [0,t_k]}\widetilde{P}_k(u)\widetilde{P}_k(u_x \cdot \px^2u)dxdt\Big{|}\\
& =: I+II.
\end{align*}
We further decompose $I$ as follows:
\begin{align*}
I &\lesssim 2^{2sk}\sum_{k_1 \le k-10}\Big{|}\int_{\R \times [0,t_k]}\widetilde{P}_k(u)\widetilde{P}_k(\widetilde{P}_{k_1}(u) \cdot \px^3u)dxdt\Big{|} \\
&+ 2^{2sk}\sum_{k_1 \ge k-9, k_2 \in \Z_+}\Big{|}\int_{\R \times [0,t_k]}\widetilde{P}_k^2(u)\widetilde{P}_{k_1}(u) \cdot \px^3\widetilde{P}_{k_2}(u)dxdt\Big{|}\\
&:=I_1 +I_2.
\end{align*}
Using \eqref{eq:energy 3} then we get that
\begin{align*}
I_1 &\lesssim 2^{2sk}\sum_{k_1 \le k-10}2^{\frac12k_1}\norm{\widetilde{P}_{k_1}(u)}_{F_{k_1}(T)}\sum_{|k-k'| \le 10 }\norm{\widetilde{P}_{k'}(u)}_{F_{k'}(T)}^2\\
&\lesssim 2^{2sk}\norm{\widetilde{P}_k(u)}_{F_k(T)}^2\norm{u}_{F^{1/2+}(T)},
\end{align*}
the last inequality comes from Cauchy-Schwarz inequality.\\
For $I_2$, using \eqref{eq:energy 2.1} and \eqref{eq:energy 2.2} then we get
\begin{align*}
I_2 &\lesssim 2^{2sk}\sum_{\substack{|k_1-k_2| \le 10 \\ k_1 \ge k+10}}2^{3k_2}2^{-2k_{max}-\frac12k_{min}} \norm{\widetilde{P}_k(u)}_{F_k(T)}\norm{\widetilde{P}_{k_1}(u)}_{F_{k_1}(T)}\norm{\widetilde{P}_{k_2}(u)}_{F_{k_2}(T)} \\
&+2^{2sk}\sum_{\substack{|k_1-k| \le 10 \\ |k_2 -k|\le10}}2^{3k_2}2^{-\frac74k_{max}}\norm{\widetilde{P}_k(u)}_{F_k(T)}\norm{\widetilde{P}_{k_1}(u)}_{F_{k_1}(T)}\norm{\widetilde{P}_{k_2}(u)}_{F_{k_2}(T)}\\
&+2^{2sk}\sum_{\substack{|k_1-k| \le 10 \\ k_2 \le k-10}}2^{3k_2}2^{-2k_{max}-\frac12k_{min}} \norm{\widetilde{P}_k(u)}_{F_k(T)}\norm{\widetilde{P}_{k_1}(u)}_{F_{k_1}(T)}\norm{\widetilde{P}_{k_2}(u)}_{F_{k_2}(T)}\\
&=:I_{2,1}+I_{2,2}+I_{2,3}.
\end{align*}
For $I_{2,1}$, since $k \le k_1$, by the Cauchy-Schwarz inequality
\begin{align*}
I_{2,1} &\lesssim 2^{sk}\sum_{\substack{|k_1-k_2| \le 10 \\ k_1 \ge k+10}}2^{sk_1}2^{(
1+\del)k_2}2^{-\del k_1/2}2^{-\del k_2/2}2^{-k/2}\norm{\widetilde{P}_k(u)}_{F_k(T)}\norm{\widetilde{P}_{k_1}(u)}_{F_{k_1}(T)}\norm{\widetilde{P}_{k_2}(u)}_{F_{k_2}(T)} \\
&\lesssim 2^{sk}2^{-(1/2+\del) k} \norm{u}_{F^s(T)}\norm{u}_{F^{1+}(T)}\norm{\widetilde{P}_k(u)}_{F_k(T)}.
\end{align*}
Similarly as above, we obtain
$$I_{2,2} \lesssim 2^{2sk}\norm{u}_{F^{5/4}(T)}\sum_{|k'-k| \le 10}\norm{\widetilde{P}_k'(u)}_{F_k'(T)}^2,$$
and
$$I_{2,3} \lesssim 2^{sk}2^{-\del k/2} \norm{u}_{F^s(T)}\norm{u}_{F^{1/2+}(T)}\norm{\widetilde{P}_k(u)}_{F_k(T)},$$
which implies that the summation on $k$ of $I$ is bounded by $\norm{u}_{F^{5/4}(T)}\norm{u}_{F^s(T)}^2$.\\
For $II$, using the same method as $I$ and \eqref{eq:energy 9}, we have
$$\sum_{k \ge 1}II \lesssim \norm{u}_{F^{5/4}(T)}\norm{u}_{F^s(T)}^2.$$
Therefore, we complete the proof of the proposition.
\end{proof}
\begin{remark}\label{re:cor1.2-energy}
To get the energy estimates for trilinear term $u^2\px u$ in \eqref{eq:5kdv heirarchy}, from \eqref{eq:energy 1}, we need to control
\begin{align*}
\sum_{k \ge 1}2^{4k}\Big| \int_{\R \times [0,t_k]} \wt{P}_k(u)&\wt{P}_k(u^2 \px u)dxdt\Big| \lesssim \sum_{k \ge 1}2^{4k} \sum_{k_1 \le k-10}\Big| \int_{\R \times [0,t_k]} \wt{P}_k(u)\wt{P}_k(\wt{P}_{k_1}(u^2) \px u)dxdt\Big| \\
&+\sum_{k \ge 1}2^{4k} \sum_{k_1 \ge k-9}\sum_{k_2 \ge 1}\Big| \int_{\R \times [0,t_k]} \wt{P}_k(u)\wt{P}_{k_1}(u^2) \wt{P}_{k_2}(\px u)dxdt\Big| =: I + II.
\end{align*}
In view of the proof of Lemma \ref{lem:energy} and Proposition \ref{prop:energy1}, it is not difficult to obtain
$$I+II \lesssim \norm{u}_{F^2}^4,$$
using \eqref{eq:cor1.2-bilinear}.
\end{remark}

\begin{proposition}\label{prop:energy2}
  Assume $s \ge 2$. Let $u,v \in F^{s}(1)$ be solutions to \eqref{eq:5kdv} with small initial data $u_0,v_0 \in H^{\infty}$ in the sense of
  $$ \norm{u_0}_{H^s} + \norm{v_0}_{H^s} \le \epsilon \ll 1 $$
  Then we have
  \begin{equation}\label{eq:energy 11}
    \norm{u-v}_{F^0(1)} \lesssim \norm{u_0-v_0}_{L^2}
  \end{equation}
and
  \begin{equation}\label{eq:energy 12}
  \norm{u-v}_{F^s(1)} \lesssim \norm{u_0-v_0}_{H^s} + \norm{u_0}_{H^{2s}}\norm{u_0-v_0}_{L^2}.
  \end{equation}
\end{proposition}

\begin{proof}
We prove first \eqref{eq:energy 11}. Since $\norm{u_0}_{H^s}+\norm{v_0}_{H^s}\le \epsilon \ll 1$, we assume from \eqref{eq:preresult} that
\begin{equation}\label{eq:energy 13}
\norm{u}_{F^s(1)} \ll 1, \norm{u}_{F^s(1)} \ll 1
\end{equation}
Let $w=u-v$, then $w$ solves the equation
\begin{equation}\label{eq:energy 14}
    \begin{cases}
\pt w + \px^5w + c_1'\px(\px w\px(u+v)) + c_2'\px(v\px^2w+w\px^2u)  \\
w(0,x) =w_0(x)=u_0(x)-v_0(x)
    \end{cases}
\end{equation}
From the linear and bilinear estimates, we obtain
\begin{equation}\label{eq:energy 15}
\begin{cases}
\norm{w}_{F^0(1)} \lesssim \norm{w}_{E^0(1)}+\norm{\px(\px w\px(u+v))}_{N^0(1)}+\norm{\px(v\px^2w+w\px^2u)}_{N^0(1)}\\
\norm{\px(\px w\px(u+v))}_{N^0(1)}+\norm{\px(v\px^2w+w\px^2u)}_{N^0(1)}  \lesssim \norm{w}_{F^0(1)}(\norm{u}_{F^s(1)}+ \norm{v}_{F^s(1)}).
\end{cases}
\end{equation}
We now devote to derive the estimate on $\norm{w}_{E^0(1)}$. From
$$\norm{w}_{E^0(1)}^2 - \norm{w_0}_{L^2}^2 \lesssim \sum_{k \ge 1} \sup_{t_k}\norm{w(t_k)}_{L^2}^2$$
and  \eqref{eq:energy 1}, we need to control
\begin{equation}\label{eq:energy 0}
\begin{split}
\sum_{k\ge1}\norm{w(t_k)}_{L^2}^2 &\lesssim \sum_{k\ge1}\Big{|}\int_{\R \times [0,t_k]}\widetilde{P}_k(w)\widetilde{P}_k(w\px^3u)dxdt\Big{|}\\
&+ \sum_{k\ge1}\Big{|}\int_{\R \times [0,t_k]}\widetilde{P}_k(w)\widetilde{P}_k(v\px^3w)dxdt\Big{|}\\
&+ \sum_{k\ge1}\Big{|}\int_{\R \times [0,t_k]}\widetilde{P}_k(w)\widetilde{P}_k(\px w\px^2u)dxdt\Big{|}\\
&+ \sum_{k\ge1}\Big{|}\int_{\R \times [0,t_k]}\widetilde{P}_k(w)\widetilde{P}_k(\px v\px^2w)dxdt\Big{|}\\
&=:I+II+III+IV.
\end{split}
\end{equation}
Using \eqref{eq:energy 2.1} and \eqref{eq:energy 2.2}, $I$ is bounded by
\begin{equation}\label{eq:energy 16}
\begin{split}
&\sum_{k \ge 1}\sum_{|k-k_1|\le5}\sum_{k_2 \le k-10} 2^{-2k_{max}-\frac12k_{min}}2^{3k_2}\norm{w}_{F_k(1)}\norm{w}_{F_{k_1}(1)}\norm{u}_{F_{k_2}(1)}\\
+&\sum_{k \ge 1}\sum_{|k-k_2|\le5}\sum_{k_1 \le k-10} 2^{-2k_{max}-\frac12k_{min}}2^{3k_2}\norm{w}_{F_k(1)}\norm{w}_{F_{k_1}(1)}\norm{u}_{F_{k_2}(1)}\\
+&\sum_{k \ge 1}\sum_{|k-k_2|\le5}\sum_{|k-k_1|\le 5} 2^{-\frac74k_{max}}2^{3k_2}\norm{w}_{F_k(1)}\norm{w}_{F_{k_1}(1)}\norm{u}_{F_{k_2}(1)}\\
+&\sum_{k \ge 1}\sum_{k \le k_1 -10}\sum_{|k_1-k_2|\le5} 2^{-2k_{max}-\frac12k_{min}}2^{3k_2}\norm{w}_{F_k(1)}\norm{w}_{F_{k_1}(1)}\norm{u}_{F_{k_2}(1)}
\end{split}
\end{equation}
For the first part of \eqref{eq:energy 16}, since $2^{k_{max}} \gtrsim 2^{k_2}$, using the Cauchy-Schwarz inequality, we have
$$\sum_{k \ge 1} \norm{w}_{F_k(1)}^2\norm{u}_{F^{1/2+}(1)}\Big{(}\sum_{k_2 \ge 0}2^{-2\del k_2}\Big{)}^{1/2} \lesssim \norm{w}_{F^0(1)}^2\norm{u}_{F^s(1)}.$$
For the last term of \eqref{eq:energy 16}, since $k+10 \le k_1, k_2 $, by the Cauchy-Schwarz inequality we have
$$\norm{u}_{F^{1+}(1)}\sum_{k,k_1 \ge 1}2^{- \frac12k}\norm{w}_{F_k(1)}2^{-\del k_1}\norm{w}_{F_{k_1}(1)} \lesssim \norm{w}_{F^0(1)}^2\norm{u}_{F^s(1)}. $$
For the second and the third parts of \eqref{eq:energy 16}, similarly as above we get
$$\sum_{k_1 \ge 1}2^{-\del k_1}\norm{w}_{F_{k_1}(1)}\norm{w}_{F^0(1)}\norm{u}_{F^{1+}(1)} \lesssim \norm{w}_{F^0(1)}^2\norm{u}_{F^s(1)},$$
and
\begin{equation}\label{eq:s=5/4}
\sum_{k \ge 1}\sum_{k-5 \le k_1, k_2 \le k+5}2^{\frac54k_2}\norm{w}_{F_{k_1}(1)}\norm{w}_{F_k(1)}\norm{u}_{F_{k_2}(1)} \lesssim \norm{w}_{F^0(1)}^2\norm{u}_{F^s(1)}.
\end{equation}
For $II$, using Lemma \ref{lem:energy} again, $II$ is dominated by
\begin{equation}\label{eq:energy 17}
\begin{split}
&\sum_{k\ge 1}\sum_{k_1 \le k-10}2^{k_1/2}\norm{v}_{F_{k_1}(1)}\norm{w}_{F_k(1)}^2\\
+&\sum_{k\ge1}\sum_{|k_1-k| \le 5}\sum_{k_2 \le k-10}2^{3k_2}2^{-2k_{max}-\frac12k_{min}} \norm{w}_{F_k(1)}\norm{v}_{F_{k_1}(1)}\norm{w}_{F_{k_2}(1)}\\
+&\sum_{k\ge1}\sum_{|k_1-k_2| \le 5}\sum_{k\le k_1-10}2^{3k_2}2^{-2k_{max}-\frac12k_{min}} \norm{w}_{F_k(1)}\norm{v}_{F_{k_1}(1)}\norm{w}_{F_{k_2}(1)}\\
+&\sum_{k\ge1}\sum_{|k-k_1| \le 5}\sum_{|k-k_2| \le 5}2^{3k_2}2^{-\frac74k_{max}} \norm{w}_{F_k(1)}\norm{v}_{F_{k_1}(1)}\norm{w}_{F_{k_2}(1)}.
\end{split}
\end{equation}
From the Cauchy-Schwarz inequality, the bound of the first term of \eqref{eq:energy 17} is easily obtained. For the second and third terms of \eqref{eq:energy 17}, since $2^{k_2} \lesssim 2^{k_1}$ and $2^{k_1}\sim 2^{k_{max}} \sim 2^{k_{med}}$, it is bounded by
\begin{align*}
&\sum_{k\ge1}\sum_{k_1\ge 0}\sum_{k_2 \ge 0}2^{-\del k_1}2^{(1+\del)k_1}2^{-\frac12k_{min}}\norm{v}_{F_{k_1}(1)}\norm{w}_{F_{k_2}(1)}\norm{w}_{F_k(1)} \\
&\lesssim \norm{v}_{F^{1+}(1)}\sum_{k\ge1}\sum_{k_2 \ge 0}2^{- \delta k/2}2^{-\del k_2/2}\norm{w}_{F_{k_2}(1)}\norm{w}_{F_k(1)} \\
&\lesssim \norm{w}_{F^0(1)}^2\norm{v}_{F^{s}(1)}.
\end{align*}
The estimate of the rest term is similar to \eqref{eq:s=5/4}.\\
Similarly to $I,II$, we can get
$$III+IV \lesssim \norm{w}_{F^0(1)}^2(\norm{u}_{F^s(1)}+\norm{v}_{F^s(1)}).$$
Therefore, we obtain the following estimate
$$\norm{w}_{E^0(1)}^2 \lesssim \norm{w_0}_{L^2}^2 + \norm{w}_{F^0(1)}^2(\norm{u}_{F^s(1)}+\norm{v}_{F^s(1)}),$$
hence, combined with \eqref{eq:energy 15} and \eqref{eq:energy 13} we obtain \eqref{eq:energy 11}. \\
Now we prove \eqref{eq:energy 12}. From the linear and bilinear estimates, we obtain
\begin{equation}\label{eq:energy 18}
\begin{cases}
\norm{w}_{F^s(1)} \lesssim \norm{w}_{E^s(1)}+\norm{\px(\px w\px(u+v))}_{N^s(1)}+\norm{\px(v\px^2w+w\px^2u)}_{N^s(1)}\\
\norm{\px(\px w\px(u+v))}_{N^s(1)}+\norm{\px(v\px^2w+w\px^2u)}_{N^0(1)}  \lesssim \norm{w}_{F^s(1)}(\norm{u}_{F^s(1)}+ \norm{v}_{F^s(1)}).
\end{cases}
\end{equation}
Since $\norm{P_{\le0}(w)}_{E^s(1)}=\norm{P_{\le0}(w_0)}_{L^2}$, it follows from \eqref{eq:energy 18} and \eqref{eq:energy 13} that
\begin{equation}\label{eq:energy 19}
\norm{w}_{F^s(1)} \lesssim \norm{w_0}_{H^s} + \norm{P_{\ge 1}(w)}_{E^s(1)}.
\end{equation}
To bound $\norm{P_{\ge 1}(w)}_{E^s(1)}$, we observe that
$$\norm{P_{\ge 1}(w)}_{E^s(1)}=\norm{P_{\ge 1}(\La^s w)}_{E^0(1)},$$
where $\La^s$ is the Fourier multiplier operator with the symbol $|\xi|^s$. Thus we apply the operator $\La^s$ on both side of the \eqref{eq:energy 14} and get
\begin{align*}
\pt\La^sw + \px^5\La^sw = &-c_1\La^sw\px^3u + -c_1\La^sv\px^3w\\
&-c_2\La^s\px w\px^2u + -c_2\La^s\px v\px^2w.
\end{align*}
We rewrite the nonlinearity in the following way:
\begin{align*}
&c_1[\La^s,w]\px^3u+c_1w\La^s\px^3u+c_1[\La^s,v]\px^3w+c_1v\La^s\px^3w\\
&c_2[\La^s,\px w]\px^2u+c_2\px w\La^s\px^2u+c_2[\La^s,\px v]\px^2w+c_2\px v\La^s\px^2w.
\end{align*}
We write the equation for $U=P_{\ge-10}(\La^sw)$ in the form
\begin{equation}\label{eq:energy 20}
\begin{cases}
\pt U + \px^5U = P_{\ge-10}(-c_1v\px^3)+P_{\ge-10}(-c_2\px v\px^2U)+P_{\ge-10}(G)+P_{\ge-10}(H)\\
U(0)=P_{\ge-10}(\La^sw_0)
\end{cases}
\end{equation}
where
\begin{align*}
G=&-c_1P_{\ge-10}(v)\La^s\px^3P_{\le-11}(w)-c_1P_{\le-11}(v)\La^s\px^3P_{\le-11}(w)\\
&-c_1[\La^s,w]\px^3u-c_1[\La^s,v]\px^3w -c_1w\La^s\px^3u,
\end{align*}
and
\begin{align*}
H=&-c_2P_{\ge-10}(\px v)\La^s\px^2P_{\le-11}(w)-c_2P_{\le-11}(\px v)\La^s\px^2P_{\le-11}(w)\\
&-c_2[\La^s,\px w]\px^2u-c_2[\La^s,\px v]\px^2w -c_2\px w\La^s\px^2u.
\end{align*}
It follows from \eqref{eq:energy 1} and \eqref{eq:energy 20} that
\begin{align*}
\norm{U}_{E^0(1)}^2-\norm{w_0}_{H^s}^2 &\lesssim \sum_{k\ge1}\Big{|}\int_{\R \times [0,t_k]}\widetilde{P}_k(U)\widetilde{P}_k(v\px^3U)dxdt\Big{|}\\
&+\sum_{k\ge1}\Big{|}\int_{\R \times [0,t_k]}\widetilde{P}_k(U)\widetilde{P}_k(G)dxdt\Big{|}\\
&+\sum_{k\ge1}\Big{|}\int_{\R \times [0,t_k]}\widetilde{P}_k(U)\widetilde{P}_k(\px v\px^2U)dxdt\Big{|}\\
&+\sum_{k\ge1}\Big{|}\int_{\R \times [0,t_k]}\widetilde{P}_k(U)\widetilde{P}_k(H)dxdt\Big{|}\\
&=:I+II+III+IV.
\end{align*}
First, consider $I$ and $III$. We can bound $I,III$ as in \eqref{eq:energy 0} and get that
$$I+III \lesssim \norm{U}_{F^0(1)}^2\norm{v}_{F^s(1)}.$$
For $II$, we estimate
\begin{align*}
II &\lesssim \sum_{k\ge1}\Big{|}\int_{\R \times [0,t_k]}\widetilde{P}_k(U)\widetilde{P}_k(P_{\ge-10}(v)\La^s\px^3P_{\le-11}(w))dxdt\Big{|}\\
&+ \sum_{k\ge1}\Big{|}\int_{\R \times [0,t_k]}\widetilde{P}_k(U)\widetilde{P}_k([\La^s,w]\px^3u)dxdt\Big{|}\\
&+ \sum_{k\ge1}\Big{|}\int_{\R \times [0,t_k]}\widetilde{P}_k(U)\widetilde{P}_k([\La^s,v]\px^3w)dxdt\Big{|}\\
&+ \sum_{k\ge1}\Big{|}\int_{\R \times [0,t_k]}\widetilde{P}_k(U)\widetilde{P}_k(w\La^s\px^3u)dxdt\Big{|}\\
&=:II_1+II_2+II_3+II_4.
\end{align*}
For $II_1$, since the derivatives fall on the law frequency, we get
\begin{align*}
II_1 &\lesssim \sum_{|k-k_1| \le 5}\sum_{k_2 \le k-10}2^{-2k_{max}-\frac12k_{min}}\norm{U}_{F_k(1)}\norm{v}_{F_{k_1}(1)}\norm{\La^sw}_{F_{k_2}(1)}\\
&+ \sum_{|k-k_2| \le 5}\sum_{k_1 \le k-10}2^{-2k_{max}-\frac12k_{min}}\norm{U}_{F_k(1)}\norm{v}_{F_{k_1}(1)}\norm{\La^sw}_{F_{k_2}(1)}\\
&+ \sum_{|k_1-k_2| \le 5}\sum_{k \le k_2-10}2^{-2k_{max}-\frac12k_{min}}\norm{U}_{F_k(1)}\norm{v}_{F_{k_1}(1)}\norm{\La^sw}_{F_{k_2}(1)}\\
&+ \sum_{|k-k_1| \le 5}\sum_{|k-k_2| \le 5}2^{-\frac74k_{max}}\norm{U}_{F_k(1)}\norm{v}_{F_{k_1}(1)}\norm{\La^sw}_{F_{k_2}(1)}\\
&\lesssim \norm{U}_{F^0(1)}^2\norm{v}_{F^s(1)},
\end{align*}
which comes from Lemma \ref{lem:energy} and the Cauchy-Schwarz inequality.\\
We consider now $II_4$. Using Lemma \ref{lem:energy} again,
\begin{align*}
II_4 &\lesssim \sum_{k\ge1}\sum_{k_1,k_2 \ge0}\Big{|}\int_{\R \times [0,t_k]}\widetilde{P}_k(U)\widetilde{P}_{k_1}(w)\La^s\px^3\widetilde{P}_{k_2}(u)dxdt\Big{|} \\
&\lesssim \sum_{|k-k_2|\le5}\sum_{k_1 \le k-10}2^{-2k_{max}-\frac12k_{min}} 2^{(s+3)k_2}\norm{\widetilde{P}_kU}_{F_k(1)}\norm{\widetilde{P}_{k_1}(w)}_{F_{k_1}(1)}\norm{\widetilde{P}_{k_2}(u)}_{F_{k_2}(1)}\\
&+ \sum_{|k-k_1|\le5}\sum_{k_2 \le k-10}2^{-2k_{max}-\frac12k_{min}} 2^{(s+3)k_2}\norm{\widetilde{P}_kU}_{F_k(1)}\norm{\widetilde{P}_{k_1}(w)}_{F_{k_1}(1)}\norm{\widetilde{P}_{k_2}(u)}_{F_{k_2}(1)}\\
&+ \sum_{|k_1-k_2|\le5}\sum_{k \le k_-10}2^{-2k_{max}-\frac12k_{min}} 2^{(s+3)k_2}\norm{\widetilde{P}_kU}_{F_k(1)}\norm{\widetilde{P}_{k_1}(w)}_{F_{k_1}(1)}\norm{\widetilde{P}_{k_2}(u)}_{F_{k_2}(1)}\\
&+ \sum_{|k-k_1|\le5}\sum_{|k-k_2|\le5}2^{-\frac74k_{max}}2^{(s+3)k_2}\norm{\widetilde{P}_kU}_{F_k(1)}\norm{\widetilde{P}_{k_1}(w)}_{F_{k_1}(1)}\norm{\widetilde{P}_{k_2}(u)}_{F_{k_2}(1)}.
\end{align*}
For the first term above, since $2^{-2k_{max}-\frac12k_{min}}2^{(3+s)k_2} \lesssim 2^{(1+s+\del)k_2}2^{-\del k_2/2}2^{-\del k/2}2^{-\frac12k_1}$, we have the bound
$$\norm{U}_{F^0(1)} \norm{w}_{F^0(1)}\norm{u}_{F^{2s}(1)}.$$
Otherwise, as $2^{k_2} \lesssim 2^{k_1}$ implies
\begin{equation}\label{eq:estimate1}
2^{-2k_{max}-\frac12k_{min}}2^{(3+s)k_2} \lesssim 2^{sk_1}2^{5k_2/4}2^{-k_{max}/4}
\end{equation}
or
\begin{equation}\label{eq:estimate2}
2^{-\frac74k_{max}}2^{(s+3)k_2} \lesssim 2^{sk_1}2^{5/4k_2},
\end{equation}
the last terms above are bounded by
$$ \norm{U}_{F^0(1)}^2\norm{u}_{F^s(1)}.$$
\eqref{eq:estimate1} and \eqref{eq:estimate2} are similarly used in later inequality repeatedly.
For $II_3$, we estimate
\begin{align*}
II_3 &\lesssim \sum_{k\ge1}\sum_{k_1 \le k_2-10}\Big{|}\int_{\R \times [0,t_k]}\widetilde{P}_k(U)\widetilde{P}_k([\La^s,\widetilde{P}_{k_1}(v)]\px^3\widetilde{P}_{k_2}(w))dxdt\Big{|}\\
&+\sum_{k\ge1}\sum_{k_1 \ge k_2-9}\Big{|}\int_{\R \times [0,t_k]}\widetilde{P}_k(U)\widetilde{P}_k([\La^s,\widetilde{P}_{k_1}(v)]\px^3\widetilde{P}_{k_2}(w))dxdt\Big{|}\\
&=:II_{3,1}+II_{3,2}.
\end{align*}
We note that in the term $II_{3,2}$ the component $v$ can spare derivative and from a similar way to $II_4$ we get
\begin{align*}
II_{3,2}&\lesssim \sum_{k\ge1}\sum_{k_1\ge k_2-9}\Big{|}\int_{\R \times [0,t_k]}\La^s\widetilde{P}_k(U)\widetilde{P}_{k_1}(v)\px^3\widetilde{P}_{k_2}(w)dxdt\Big{|}\\
&+ \sum_{k\ge1}\sum_{k_1\ge k_2-9}\Big{|}\int_{\R \times [0,t_k]}\widetilde{P}_k(U)\widetilde{P}_{k_1}(v)\La^s\px^3\widetilde{P}_{k_2}(w)dxdt\Big{|}\\
&\lesssim \norm{U}_{F^0(1)}^2\norm{v}_{F^s(1)},
\end{align*}
where we used $2^{k_1} \sim 2^k \sim 2^{k_{max}}$ and a similar argument to \eqref{eq:estimate1}.\\
For $II_{3,1}$, we need to exploit the cancellation of the commutator. By taking $\gamma$ and extending $U,v,w$ as in the proof of Lemma \ref{lem:energy}, then we get
\begin{align*}
II_{3,1} \lesssim \sum_{k\ge1}\sum_{k_1 \le k_2-10}\sum_{|n|\lesssim 2^{2k_3}}&\Big{|}\int_{\R \times [0,t_k]}(\gamma(2^{2k_3}t-n)\mathbf{1}_{[0,t_k]}(t)\widetilde{P}_k(U))\\
&\times\widetilde{P}_k([\La^s,\widetilde{P}_{k_1}(\gamma(2^{2k_3}t-n)v)]\px^3\widetilde{P}_{k_2}(\gamma(2^{2k_3}t-n)w))dxdt\Big{|}.
\end{align*}
Let $f_k=\gamma(2^{2k_3}t-n)\widetilde{P}_k(U)$, $g_{k_1}=\widetilde{P}_{k_1}(\gamma(2^{2k_3}t-n)v)$ and $h_{k_2}=\widetilde{P}_{k_2}(\gamma(2^{2k_3}t-n)w))$. It is easy to see from $|k_2-k|\le 5$ that
$$|\ft([\La^s,g_{k_1}]\px^3h_{k_2})(\xi,\tau)| \lesssim \int_{\R^2}|\widehat{g}_{k_1}(\xi-\xi_1,\tau-\tau_1)|2^{3k_1}2^{sk_2}|\widehat{h}_{k_2}(\xi_1,\tau_1)|d\xi_1d\tau_1,$$
which follows similarly to \eqref{eq:energy 8}. Then using the same argument in the proof of Lemma \ref{lem:energy}, we can get that
\begin{align*}
II_{3,1} &\lesssim \sum_{k\ge1}\sum_{k_1 \le k_2-10}2^{3k_1}2^{sk_2}2^{-2k_{max}-\frac12k_{mim}}\norm{\widetilde{P}_kU}_{F_k(1)}\norm{\widetilde{P}_{k_1}(v)}_{F_{k_1}(1)}\norm{\widetilde{P}_{k_2}(w)}_{F_{k_2}(1)}\\
&\lesssim \norm{U}_{F^0(1)}^2\norm{v}_{F^s(1)}.
\end{align*}
from a similar way to \eqref{eq:estimate1}.\\
The $II_2$ is identical to the one of $II_3$ from symmetry.\\
So, we now need to control $IV$, but controlling $IV$ is similar and easier than $II$ since derivatives are distributed.\\
Hence we have proved that
$$\norm{U}_{E^0(1)}^2 \lesssim \norm{w_0}_{H^s}^2 + \norm{U}_{F^0(1)}^2(\norm{u}_{F^s(1)}+\norm{v}_{F^s(1)})+ \norm{U}_{F^0(1)}\norm{w}_{F^0(1)}\norm{u}_{F^{2s}(1)}.$$
By \eqref{eq:energy 13} and \eqref{eq:energy 11}, we get
$$\norm{U}_{E^0(1)}\lesssim \norm{u_0-v_0}_{H^s}+\norm{u_0-v_0}_{L^2}\norm{u_0}_{H^{2s}},$$
from which combined with \eqref{eq:energy 20} we completes the proof of the proposition.
\end{proof}

\section{Proof of Theorem \ref{thm:main}}
In this section, we prove Theorem \ref{thm:main}. The main ingredients are $L^2-$convolution estimates which is proved in section \ref{sec:nonlinear} and energy estimates obtained in section \ref{sec:energy}. The basic idea follows the idea of Ionescu, Kenig and Tataru \cite{IKT} and for the weighted $X_k$ norm, we refer to Guo, Peng, Wang and Wang \cite{GPWW}.
\begin{proposition}\label{prop:H^s}
Let $s\ge 0$, $T\in (0,1]$, and $u\in F^s(T)$, then
    \begin{equation}\label{eq:H^s}
\sup_{t\in [-T,T]} \|u(t)\|_{H^s} \lesssim \|u\|_{F^s(T)}
    \end{equation}
\end{proposition}

\begin{proposition}\label{prop:linear}
Let $T\in (0,1]$,  $u,v\in C([-T,T]:H^\infty)$ and
\begin{equation}\label{eq:5kdv2}
\partial_tu + \px^5u  = v  \ on \  \R\times(-T,T)
\end{equation}
Then we have
\begin{equation}\label{eq:linear}
\|u\|_{F^s(T)} \lesssim \|u\|_{E^s(T)}  + \|v\|_{N^s(T)},
\end{equation}
for any $s \ge 0$.
\end{proposition}
The proof of the Proposition \ref{prop:H^s} and \ref{prop:linear} are similar to \cite{IKT} and  \cite{GPWW}. For self-containedness, we give the proof in Appendix.\\

Now, we show the local well-posedness of \eqref{eq:5kdv} by using the classical energy method. From Duhamel's principle, we get that the equation \eqref{eq:5kdv} is equivalent to the following integral equation,
\begin{equation}\label{eq:duhamel}
u(t) = W(t)u_0 + \int_{0}^{t} W(t-s) v(s) ds,
\end{equation}
where $v(t,x) = c_1\px u\px^2u + c_2u\px^3u$.\\
We will work on the following localized version,
\begin{equation}\label{eq:localized duhamel}
u(t) = \eta_0(t)W(t)u_0 + \eta_0(t)\int_{0}^{t} W(t-s) v(s) ds.
\end{equation}
Then we see that if $u$ is a solution to \eqref{eq:localized duhamel} on $\R$, then $u$ solves \eqref{eq:duhamel} on $[-1,1]$. \\
By the scaling invariance:
\begin{equation}\label{eq:scaling}
u_{\lambda}(t,x) = \lambda^{-2}u(\frac{t}{\lambda^5}, \frac{x}{\lambda}),
\end{equation}
and observing $s_c=-\frac 32$, we may assume that
\begin{equation}\label{eq:assumption}
\norm{u_0}_{H^s} \le \epsilon \ll 1.
\end{equation}
For part (a) of Theorem \ref{thm:main}, we assume that $s \ge \frac54$.\\
We already know from \cite{Ponce94} that there is a smooth solution to the \eqref{eq:5kdv} with $u_0 \in H^{\infty}$. So, we show \emph{a priori} bound: if $T \in (0,1]$ and $u \in C([-T,T]:H^{\infty})$ is a solution of \eqref{eq:5kdv} with $\norm{u_0}_{H^s} \le \epsilon \ll 1$, then
\begin{equation}\label{eq:result}
\sup_{t \in [-T,T]}\norm{u(t)}_{H^s} \lesssim \norm{u_0}_{H^s}.
\end{equation}
It comes from the linear estimate (Proposition \ref{prop:linear}), the $L^2$ estimate (Proposition \ref{prop:bilinear} (a)) and the energy estimate (Proposition \ref{prop:energy1}). More precisely, for any $T' \in [0,T]$ we have
\begin{eqnarray}\label{eq:proof1}
\left \{
\begin{array}{l}
\norm{u}_{F^{s}(T')}\lesssim \norm{u}_{E^s(T')} + \norm{\px^3(u^2)}_{N^s(T')} + \norm{u\px^3u}_{N^s(T')};\\
\norm{\px^3(u^2)}_{N^s(T')} + \norm{u\px^3u}_{N^s(T')} \lesssim
\norm{u}_{F^s(T')}^2;\\
\norm{u}_{E^s(T')}^2 \lesssim
\norm{u_0}_{H^s}^2 + \norm{u}_{F^s(T')}^3.
\end{array}
\right.
\end{eqnarray}
Let $X(T')=\norm{u}_{E^s(T')} + \norm{\px^3(u^2)}_{N^s(T')} + \norm{u\px^3u)}_{N^s(T')}$. From similar argument as in the proof of Lemma 4.2 in \cite{IKT}, we know that $X(T')$ is continuous, increasing on $[-T,T]$ and satisfies
$$\lim_{T' \to 0}X(T') \lesssim \norm{u_0}_{H^s}.$$
Moreover, we obtain from \eqref{eq:proof1} that
$$X(T')^2 \lesssim \norm{u_0}_{H^s}^2 +X(T')^3 +X(T')^4.$$
If $\epsilon$ is small enough, then using bootstrap (see \cite{Tao3}), $X(T') \lesssim \norm{u_0}_{H^s}$ can be obtained by \eqref{eq:assumption}. Hence we obtain
\begin{equation}\label{eq:preresult}
\norm{u}_{F^s(T')} \lesssim \norm{u_0}_{H^s},
\end{equation}
which implies \eqref{eq:result} by Proposition \ref{prop:H^s}.\\

\noindent For part (b), we now assume $s \ge 2$ so that we use Proposition \ref{prop:bilinear} (b) and Proposition \ref{prop:energy2}. In order to obtain a solution in $H^s$, we use compactness argument which follows the ideas in \cite{IKT}. Fix $u_0 \in H^s$. Then we can choose $\{u_{0,n}\}_{n=1}^{\infty} \subset H^{\infty}$ such that $u_{0,n} \to u_0$ in $H^s$ as $n \to \infty$. Let $u_n(t) \in H^{\infty}$ is a solution to \eqref{eq:5kdv} with initial data $u_{0,n}$. Then it suffices to show that the sequence $\{u_n\}$ is a Cauchy sequence in $C([-T,T]:H^s)$. For $K \in \Z_+$, let $u_{0,n}^K = P_{\le K}(u_{0,n})$. Then since
$$\sup_{t \in [-T,T]}\norm{u_m-u_n}_{H^s} \lesssim \sup_{t \in [-T,T]}\norm{u_m-u_m^K}_{H^s} +\sup_{t \in [-T,T]}\norm{u_m^K-u_n^K}_{H^s} +\sup_{t \in [-T,T]}\norm{u_n^K-u_n}_{H^s},$$
it suffices to show that for any $\vep > 0 $ and $K$, we have
\begin{equation}\label{eq:proof2}
\sup_{t \in [-T,T]}\norm{u_n-u_n^K}_{H^s} \le \vep/3 \quad \mbox{and} \quad \sup_{t \in [-T,T]}\norm{u_m^K-u_n^K}_{H^s} \le \vep/3,
\end{equation}
for sufficiently large $n,m$.\\
First, choose  large $K$ such that $\norm{u_0-u_0^K}_{H^s} \le o(1)$. Then since $u_{0,n}^K \to u_0^K$ in $H^s$ for any $K$, we get
\begin{align*}
\sup_{t \in [-T,T]}\norm{u_m^K-u_n^K}_{H^s} &\lesssim \norm{u_m^K-u_n^K}_{F^s(T)}\\
&\lesssim \norm{u_{0,m}^K-u_{0,n}^K}_{H^s} + \norm{u_{0,n}^K}_{H^{2s}}\norm{u_{0,m}^K-u_{0,n}^K}_{L^2} \\
&\lesssim \norm{u_{0,m}^K-u_0^K}_{H^s} + \norm{u_{0,n}^K-u_0^K}_{H^s},
\end{align*}
for large $m,n$ and by Proposition \ref{prop:H^s} and \ref{prop:energy2}. And this gives the second part of \eqref{eq:proof2}.\\
From same argument to above and $\norm{u_{0,n}-u_0}_{H^s} \to 0$ for large $n$, we get the first part of \eqref{eq:proof2}. Hence, we complete the existence of a solution. The uniqueness of the solution and the last part of Theorem \ref{thm:main} comes from the classical energy method, the scaling \eqref{eq:scaling}, and Proposition \ref{prop:H^s}. We omit the detail.

\appendix

\section{}
In appendix, we collect proofs of Proposition~\ref{prop:H^s} and Proposition \ref{prop:linear} for convenience of readers. Similar proofs are found in \cite{IKT,Guo2, GPWW}.\\

\begin{proof}[Proof of Proposition \ref{prop:H^s}] $\,$\\
We use extended formula $\wt{u}_k$ instead of $u$ as in proof of Proposition \ref{prop:bilinear}. First, our observation is that,
\begin{equation}\label{eq:appendix a.1}
\sup_{t\in [-T,T]} \|u(t)\|_{H^s}^2 \approx \sum_{k \ge 0} 2^{2sk}\sup_{t_k \in [-T,T]}\norm{\ft_x[u_k(t_k)]}_{L^2(\R)}^2
\end{equation}
and
\begin{equation}\label{eq:appendix a.2}
\|u(t)\|_{F^s(T)}^2 \approx \sum_{k \ge 0} 2^{2sk}\sup_{t_k \in [-T,T]}\norm{\ft[\eta_0(2^k(t-t_k))u_k)]}_{X_k}^2.
\end{equation}
From comparing \eqref{eq:appendix a.1} and \eqref{eq:appendix a.2}, it is enough to prove that
\begin{equation}\label{eq:appendix a.3}
\norm{\ft_x[\wt{u}_k(t_k)]}_{L^2(\R)} \lesssim \norm{\ft[\eta_0(2^k(t-t_k))\wt{u}_k)]}_{X_k},
\end{equation}
for $k \in \Z_+$, $t_k \in [-T,T]$ and $\wt{u}_k \in F_k$, which is an extension of $u_k$. (Precise explanation of an extension will be mentioned in the proof of Proposition \ref{prop:linear} later.)\\
Let $f_k = \ft[\eta_0(2^k(t-t_k))\wt{u}_k)]$, then we have
$$\ft_x[\wt{u}_k(t_k)](\xi) = c\int_{\R}f_k(\tau,\xi)e^{it_k\tau} d\tau.$$
From \eqref{eq:pre1}, we have
$$\norm{\ft_x[\wt{u}_k(t_k)]}_{L_{\xi}^2} \lesssim \normo{\int_{\R}f_k(\tau,\xi)e^{it_k\tau} d\tau}_{L_{\xi}^2} \lesssim \norm{f_k}_{X_k},$$
which implies completion of \eqref{eq:appendix a.3} and completes the proof of the Proposition \ref{prop:H^s}.\\
\end{proof}

\begin{proof}[Proof of Proposition \ref{prop:linear}] $\,$\\
To prove this proposition, we see from the definitions that the square of the right-hand side of \eqref{eq:linear} as following,
$$\norm{P_{\le0}(u(0))}_{L^2}^2 + \norm{P_{\le0}(v))}_{N_0(T)}^2 + \sum_{k \ge 1}\Big( \sup_{t_k \in [-T,T]}  2^{2sk}\norm{P_{k}(u(t_k))}_{L^2}^2 + 2^{2sk}\norm{P_{k}(v)}_{N_k(T)}^2\Big).$$
Thus, from definitions, it is enough to prove that
\begin{equation}\label{eq:appendix b.2}
\begin{cases}
\norm{P_{\le0}(u)}_{F_0(T)} \lesssim \norm{P_{\le0}(u(0))}_{L^2} + \norm{P_{\le0}(v))}_{N_0(T)} ; \\
\norm{P_{k}(u)}_{F_k(T)} \lesssim \sup\limits_{t_k \in [-T,T]} \norm{P_{k}(u(t_k))}_{L^2} + \norm{P_{k}(v))}_{N_k(T)}  \qquad \mbox{if} \quad k \ge 1.
\end{cases}
\end{equation}
for $k \in \Z_+$ and $u,v\in C([-T,T]:H^\infty)$ which solve \eqref{eq:5kdv2}.\\
\noindent $\mathbf{Step \ 1 : Extension \ of} \ P_k(u)$.\\
Fix $k \ge 0$ and $\wt{v}$ denote an extension of $P_k(v)$ such that $\norm{\wt{v}}_{N_k} \le C\norm{v}_{N_k(T)}$. In view of \eqref{eq:Sk}, we may assume that $\wt{v}$ is supported in $\R \times [-T-2^{-2k-10},T+2^{-2k-10}]$. More precisely, let $\theta(t)$ be a smooth function such that
$$\theta(t) = 1, \quad \mbox{if} \quad t \ge 1; \qquad \quad \theta(t)=0, \quad \mbox{if} \quad t \le 0.$$
Let $m_k(t)=\theta(2^{2k+10}(t+T+2^{-2k-10}))\theta(-2^{2k+10}(t-T-2^{-2k-10}))$. Then $m_k \in S_k$ and we see that $m_k$ is supported in $[-T-2^{-2k-10},T+2^{-2k-10}]$ and equal to $1$ in $[-T,T]$. From \eqref{eq:Sk}, we consider $\wt{v}$ instead of $m_k(t)v$. We define for $t \ge T$,
$$\wt{u}(t) = \eta_0(2^{2k+5}(t-T))\bigg[ W(t-T)P_k(u(T)) + \int_{T}^{t} W(t-s)P_k(\wt{v}(s)) ds \bigg],$$
and for $t \le -T$,
$$\wt{u}(t) = \eta_0(2^{2k+5}(t+T))\bigg[ W(t+T)P_k(u(-T)) + \int_{t}^{-T} W(t-s)P_k(\wt{v}(s)) ds \bigg].$$
For $t \in [-T,T]$, we define $\wt{u}(t)=u(t)$. It is obvious that $\wt{u}$ is an extension of $u$ and we get from definitions of $F_k(T)$ and $F_k$ that
\begin{equation}\label{eq:extension1}
\norm{u}_{F_k(T)} \lesssim \sup_{t_k \in [-T,T]} \norm{\ft[\wt{u}\cdot\eta_0(2^{2k}(t-t_k))]}_{X_k}.
\end{equation}
Indeed, in view of the definition of $F_k$, we can get \eqref{eq:extension1} if the followig holds.
\begin{equation}\label{eq:extension2}
\sup_{t_k \in \R}\norm{\ft[\wt{u}\cdot\eta_0(2^{2k}(t-t_k))]}_{X_k} \lesssim \sup_{t_k \in [-T,T]}\norm{\ft[\wt{u}\cdot\eta_0(2^{2k}(t-t_k))]}_{X_k}.
\end{equation}
For $t_k > T$, since $\wt{u}$ is supported in $[-T-2^{-2k-5},T+2^{-2k-5}]$, we can see that
$$\wt{u}\eta_0(2^{2k}(t-t_k)) = \wt{u}\eta_0(2^{2k}(t-T))\eta_0(2^{2k}(t-t_k)).$$
And we get from \eqref{eq:pXk3} that
$$\sup_{t_k > T}\norm{\ft[\wt{u}\cdot\eta_0(2^{2k}(t-t_k))]}_{X_k} \lesssim \sup_{t_k \in [-T,T]}\norm{\ft[\wt{u}\cdot\eta_0(2^{2k}(t-t_k))]}_{X_k}.$$
Using the same method for $t_k < -T$, then we obtain \eqref{eq:extension1}.\\

\noindent $\mathbf{Step \ 2 : Linear \ estimates}$\\
For fixed $k \ge 0$. In view of the definitions, \eqref{eq:extension1} and \eqref{eq:pXk3}, it is enough to prove that if $\phi_k \in L^2$ with $\wh{\phi}_k$ supported in $I_k$, and $v_k \in N_k$ with time support in an interval $I$ ($|I| \lesssim 2^{2k}$), then
\begin{equation}\label{eq:appendix b.3}
\norm{\ft[u_k\cdot\eta_0(2^{2k}t)]}_{X_k} \lesssim \norm{\phi_k}_{L^2} + \norm{(\tau-w(\xi) +i2^{2k})^{-1}\ft(v_k)}_{X_k},
\end{equation}
where
\begin{equation}\label{eq:appendix b.4}
u_k(t) = W(t)\phi_k + \int_{0}^{t} W(t-s) v_k(s) ds.
\end{equation}
Then from the properties of Fourier transform, direct computations show that
\begin{equation}\label{eq:appendix b.5}
\begin{split}
\ft[u_k\cdot&\eta_0(2^{2k}t)](\tau, \xi) = \ft_x(\phi_k)(\xi)\ft_t[e^{itw(\xi)}\eta_0(2^{2k}t)](\tau)\\
&+ C\int_{\R} \ft(v_k)(\xi,\tau')\frac{\ft_t(\eta_0)(2^{-2k}(\tau-\tau'))-\ft_t(\eta_0)(2^{-2k}(\tau-w(\xi)))}{i2^{2k}(\tau'-w(\xi))} d\tau'.
\end{split}
\end{equation}
More precisely, for second part of \eqref{eq:appendix b.5}, consider $\eta_0(t)$ instead of $\eta_0(2^{2k}t)$. Then we have
\begin{align*}
\ft\Big[\eta_0(t)\int_{0}^{t} W(t-s) u(s) ds\Big] &= \ft_t\Big[\eta_0(t)e^{itw(\xi)}\int_{0}^{t}e^{-isw(\xi)}\ft_x(u)(s)ds\Big] \\
&= \ft_t\Big[\eta_0(t)e^{itw(\xi)}\int_{\R}e^{-isw(\xi)}\mathbf{1}_{[0,t]}(s)\ft_x(u)(s)ds\Big]\\
&= \ft_t\Big[\eta_0(t)e^{itw(\xi)}(\hat{u} \ast \ft_s(\mathbf{1}_{[0,t]}))(w(\xi))\Big].
\end{align*}
Since
$$\ft_s(\mathbf{1}_{[0,t]})(\tau) = \frac{e^{-it\tau}-1}{-i\tau},$$
we obtain that
\begin{align*}
\ft\Big[\eta_0(t)\int_{0}^{t} W(t-s) u(s) ds\Big] &= \ft_t\Big[\eta_0(t)e^{itw(\xi)}\int_{\R}\hat{u}(\tau',\xi) \frac{e^{-it(w(\xi)-\tau')}-1}{i(\tau'-w(\xi))} d\tau'\Big] \\
&=\int_{\R^2}e^{-it\tau}\eta_0(t)e^{itw(\xi)}\hat{u}(\tau',\xi) \frac{e^{-it(w(\xi)-\tau')}-1}{i(\tau'-w(\xi))} d\tau'dt \\
&=\int_{\R}\hat{u}(\tau',\xi) \frac{\ft_t(\eta_0)(\tau-\tau')-\ft_t(\eta_0)(\tau-w(\xi))}{i(\tau'-w(\xi))} d\tau'.
\end{align*}
Because of $\ft_t(f(\la t))(\tau) = \la^{-1}\ft_t(f(t))(\la^{-1}\tau)$, we get the second part of \eqref{eq:appendix b.5}.\\
We consider that the right-hand side of \eqref{eq:appendix b.5} separately.
\begin{lemma}\label{lem:Es}
Let $\phi_k \in L^2$ with $\phi_k$ supported in $I_k$. Then, for any $k \in \Z_+$, we have
\begin{equation}\label{eq:Es}
\norm{\ft[\eta_0(2^{2k}t)W(t)\phi_k]}_{X^k} \lesssim \norm{\phi_k}_{L^2}.
\end{equation}
\end{lemma}
\begin{proof}
Since
$$\norm{\ft[\eta_0(2^{2k}t)W(t)\phi_k]}_{X^k} = \sum_{j \ge 0}2^{j/2}\beta_{k,j}\norm{\eta_j(\tau-w(\xi))\ft_t[e^{itw(\xi)}\eta_0(2^{2k}t)]\ft_x(\phi_k)(\xi)}_{L_{\tau, \xi}^2},$$
we need to show that
\begin{equation}\label{eq:appendix b.6}
\sum_{j \ge 0}2^{j/2}\beta_{k,j}\norm{\eta_j(\tau-w(\xi))\ft_t[e^{itw(\xi)}\eta_0(2^{2k}t)]}_{L_{\tau}^2} \lesssim 1.
\end{equation}
Since $\eta_0$ is a Schwartz function, which decays faster than any polynomial, we can get \eqref{eq:appendix b.6}, which makes Lemma \ref{lem:Es} to be true.
\end{proof}
\begin{lemma}\label{lem:Ns}
Let $v_k \in N_k$. Then, for any $k \in \Z_+$,
\begin{equation}\label{eq:Ns}
\normb{\ft\Big[\int_{0}^{t} \eta_0(2^{2k}t)W(t-s) v_k(s) ds\Big]}_{X_k} \lesssim \norm{(\tau-w(\xi) +i2^{2k})^{-1}\ft(v_k)}_{X_k}.
\end{equation}
\end{lemma}
\begin{proof}
Consider the second part of \eqref{eq:appendix b.5}. From the oscillatory integrals for the smooth functions, we observe  that
\begin{align*}
\bigg| &\frac{\ft_t[\eta_0(2^{-2k}(\tau-\tau'))]-\ft_t[\eta_0(2^{-2k}(\tau-w(\xi)))]}{2^{2k}(\tau'-w(\xi))} \cdot (\tau'-w(\xi) + 2^{2k})\bigg| \\
&\lesssim 2^{-2k}(1+2^{-2k}|\tau - \tau'|)^{-4} + 2^{-2k}(1+2^{-2k}|\tau - w(\xi)|)^{-4}.
\end{align*}
So, let $\bar{v}_k = (\tau-w(\xi) +i2^{2k})^{-1}\ft(v_k)$. Then we need to show that
\begin{equation}\label{eq:appendix b.7}
\sum_{j \ge 0}2^{j/2}\beta_{k,j}\normb{\eta_j(\tau-w(\xi))\int_{\R}\bar{v}_k(\xi, \tau')2^{-2k}(1+2^{-2k}|\tau - \tau'|)^{-4} d\tau'}_{L^2} \lesssim \norm{\bar{v}_k}_{X_k}
\end{equation}
and
\begin{equation}\label{eq:appendix b.8}
\sum_{j \ge 0}2^{j/2}\beta_{k,j}\normb{\eta_j(\tau-w(\xi))\int_{\R}\bar{v}_k(\xi, \tau')2^{-2k}(1+2^{-2k}|\tau - w(\xi)|)^{-4} d\tau'}_{L^2} \lesssim \norm{\bar{v}_k}_{X_k}.
\end{equation}
For \eqref{eq:appendix b.8}, we observe that
\begin{align*}
\normb{\eta_j(\tau-w(\xi))\int_{\R}&\bar{v}_k(\xi, \tau')2^{-2k}(1+2^{-2k}|\tau - w(\xi)|)^{-4} d\tau'}_{L^2}\\
&\lesssim \norm{\eta_j(\tau-w(\xi))2^{-2k}(1+2^{-2k}|\tau - w(\xi)|)^{-4} }_{L_{\tau}^2}\normb{\int_{\R}\bar{v}_k(\xi, \tau') d\tau'}_{L_{\xi}^2}.
\end{align*}
Since $2^{-2k}(1+2^{-2k}|\tau - w(\xi)|)^{-4}$ decays faster than $\beta_{k,j}(1+|\tau - w(\xi)|^{2})^{1/4}$, we have
$$\sum_{j \ge 0}2^{j/2}\beta_{k,j}\norm{\eta_j(\tau-w(\xi))2^{-2k}(1+2^{-2k}|\tau - w(\xi)|)^{-4} }_{L_{\tau}^2} \lesssim 1.$$
So, we can say from \eqref{eq:pre1} that \eqref{eq:appendix b.8} is true.\\
For \eqref{eq:appendix b.7}, assume that $k \ge 1$. Then, for $j \le 2k$, since $\beta_{k,j} \sim 1$, we have from Cauchy-Schwarz inequality and \eqref{eq:pre1} that
\begin{align*}
\sum_{j \le 2k}2^{j/2}\beta_{k,j}&\Big\|\eta_j(\tau-w(\xi))\int_{\R}\bar{v}_k(\xi, \tau')2^{-2k}(1+2^{-2k}|\tau - \tau'|)^{-4} d\tau'\Big\|_{L^2} \\
&\lesssim 2^{2k/2}\Big\|\eta_{\le2k}(\tau-w(\xi))\int_{\R}\bar{v}_k(\xi, \tau')2^{-2k}(1+2^{-2k}|\tau - \tau'|)^{-4} d\tau'\Big\|_{L^2} \lesssim \norm{\bar{v}_k}_{X_k}.
\end{align*}
For $2k < j \le 5k$, we also have $\beta_{k,j} \sim 1$. Thus, we have from \eqref{eq:pre1} that
\begin{align*}
\sum_{2k < j \le 5k}\sum_{j_1 \ge 0}&2^{j/2}\Big\|\eta_j(\tau-w(\xi))\int_{\R}\bar{v}_{k,j_1}(\xi, \tau')2^{-2k}(1+2^{-2k}|\tau - \tau'|)^{-4} d\tau'\Big\|_{L^2} \\
&\lesssim \sum_{2k < j \le 5k}\sum_{j_1 \ge 0}2^{j}2^{6k - 4\max(j,j_1)}\Big\|\int_{\R}\bar{v}_{k,j_1}(\xi, \tau')d\tau'\Big\|_{L_{\xi}^2} \lesssim \norm{\bar{v}_k}_{X_k}.
\end{align*}
For the rest term $(j > 5k)$, similarly as before, we get
\begin{align*}
\sum_{5k < j }\sum_{j_1 \ge 0}&2^{j/2}\beta_{k,j}\Big\|\eta_j(\tau-w(\xi))\int_{\R}\bar{v}_{k,j_1}(\xi, \tau')2^{-2k}(1+2^{-2k}|\tau - \tau'|)^{-4} d\tau'\Big\|_{L^2} \\
&\lesssim \sum_{5k < j }\sum_{j_1 \ge 0}2^{j}2^{(j-5k)/8}2^{6k - 4\max(j,j_1)}\Big\|\int_{\R}\bar{v}_{k,j_1}(\xi, \tau')d\tau'\Big\|_{L_{\xi}^2} \lesssim \norm{\bar{v}_k}_{X_k}.
\end{align*}
Now consider $k=0$. In this case, we can easily derive than $k \ge 1 $ case. Similarly as before, we have
\begin{align*}
\sum_{j \ge 0}\sum_{j_1 \ge 0}&2^{j/2}\beta_{0,j}\Big\|\eta_j(\tau-w(\xi))\int_{\R}\bar{v}_{0,j_1}(\xi, \tau')(1+|\tau - \tau'|)^{-4} d\tau'\Big\|_{L^2} \\
&\lesssim \sum_{j \ge 0 }\sum_{j_1 \ge 0}2^{3j/2}2^{- 4\max(j,j_1)}\Big\|\int_{\R}\bar{v}_{0,j_1}(\xi, \tau')d\tau'\Big\|_{L_{\xi}^2} \lesssim \norm{\bar{v}_0}_{X_0}.
\end{align*}
Hence, we proved Lemma \ref{lem:Ns} completely.
\end{proof}
\noindent Combining \eqref{eq:Es} and \eqref{eq:Ns} completes the proof of Propostion \ref{prop:linear}.
\end{proof}

\end{document}